\documentclass{amsart}

\usepackage{amsthm,amsfonts,amsmath,amssymb}
\usepackage[all]{xy}
\usepackage{lscape}
\usepackage{hyperref}
\usepackage{pdfpages}
\usepackage{verbatim}
\usepackage[normalem]{ulem}
\theoremstyle{plain}

\newtheorem{theorem*}{Theorem}

\newtheorem{theorem}{Theorem}[section]
\newtheorem{lemma}[theorem]{Lemma}     
\newtheorem{corollary}[theorem]{Corollary}
\newtheorem{proposition}[theorem]{Proposition}

\theoremstyle{definition}

\newtheorem{definition}[theorem]{Definition}

\theoremstyle{remark}

\newtheorem{remark}[theorem]{Remark}

\DeclareMathOperator{\rk}{rk}
\DeclareMathOperator{\card}{card}
\DeclareMathOperator{\Ort}{Ort}
\DeclareMathOperator{\Span}{span}

\DeclareMathOperator{\Hom}{Hom}

\DeclareMathOperator{\height}{ht}
\DeclareMathOperator{\supp}{supp}

\DeclareMathOperator{\Stab}{Stab}
\DeclareMathOperator{\ad}{ad}

 \newcommand{\calF}{\mathcal F}
 
\newcommand{\calI}{\mathcal I}

 \newcommand{\calN}{\mathcal N}
\newcommand{\calO}{\mathcal O} 
\newcommand{\calX}{\mathcal X}

 \newcommand{\mN}{\mathbb N}
 \newcommand{\mQ}{\mathbb Q}
 \newcommand{\mZ}{\mathbb Z}

\newcommand{\goa}{\mathfrak a}
\newcommand{\goc}{\mathfrak c}
\newcommand{\gog}{\mathfrak g}
\newcommand{\gol}{\mathfrak l}

\newcommand{\goh}{\mathfrak h}

\newcommand{\gou}{\mathfrak u}

\newcommand{\gob}{\mathfrak b}
\newcommand{\gop}{\mathfrak p}

\newcommand{\got}{\mathfrak t}

\newcommand{\gra}{\alpha} \newcommand{\grb}{\beta}    \newcommand{\grg}{\gamma}
\newcommand{\grd}{\delta} \newcommand{\grl}{\lambda}  \newcommand{\grs}{\sigma}
\newcommand{\gre}{\varepsilon} \newcommand{\gro}{\omega}

 \newcommand{\grD}{\Delta}  \newcommand{\grL}{\Lambda}

\newcommand{\mk}  {\Bbbk}

\renewcommand{\setminus}      {\smallsetminus}

\renewcommand{\geq}      {\geqslant}
\renewcommand{\leq}      {\leqslant}

\newcommand{\ol}         {\overline}

\newcommand{\wt}         {\widetilde}
\newcommand{\wh}         {\widehat}

\title[Nilpotent orbits of height 2]{Nilpotent orbits of height 2\\
and involutions in the affine Weyl group}

\author[J. Gandini]{Jacopo Gandini}
\address{Dipartimento di Matematica, Universit\`a di Bologna, Piazza di Porta San Donato 5, 40126 Bologna, Italy}
\email{jacopo.gandini@unibo.it}

\author[P. M\"oseneder Frajria]{Pierluigi M\"oseneder Frajria}
\address{Politecnico di Milano, Polo regionale di Como, 
Via Valleggio 11, 22100 Como, Italy}
\email{pierluigi.moseneder@polimi.it}

\author[P. Papi]{Paolo Papi}
\address{Dipartimento di Matematica, Sapienza Universit\`a di Roma, P.le A. Moro 2, 00185 Roma, Italy}
\email{papi@mat.uniroma1.it}

\date{\today}

\begin{document}

\begin{abstract} Let $G$ be an almost simple group over an algebraically closed field $\mk$ of characteristic zero, let $\gog$ be its Lie algebra and  let $B \subset G$ be a Borel subgroup. Then $B$ acts with finitely many orbits on the variety $\calN_2 \subset \gog$ of the nilpotent elements whose height is at most 2. We provide a parametrization of the $B$-orbits in $\calN_2$ in terms of subsets of pairwise orthogonal roots, and we provide a complete description of the inclusion order among the $B$-orbit closures in terms of the Bruhat order on certain involutions in the affine Weyl group of $\gog$.
\end{abstract}

\maketitle

\section*{Introduction.}

Let $G$ be an almost simple group over an algebraically closed field $\mk$ of characteristic zero, let $\gog$ be its Lie algebra and let $\calN \subset \gog$ be the nilpotent cone. It is well known that $G$ acts with finitely many orbits on $\calN$: for instance, when $G$ is a classical group, the nilpotent $G$-orbits are parametrized in terms of partitions, and the partial order defined by the inclusions of their closures is nicely expressed in terms of the dominance order of partitions.

Given $e \in \calN$ a natural index of nilpotency is the \textit{height}, defined as
$$
	\height(e) = \max \{n \in \mN \; | \; \ad(e)^n \neq 0\}.
$$
The simplest nonzero nilpotent elements are those of height 2. If $G$ is a special linear or a symplectic group, these are precisely the nilpotent elements whose square is zero. More generally, if $G$ is a classical group and $e$ is a nonzero nilpotent element with  $e^2 = 0$ then  $e$ has height 2, however if $G$ is an orthogonal group and $\height(e) = 2$ it might also be $e^2 \neq 0$.

Beyond the height 2 elements, a nice class of nilpotent elements of small height is that of the \textit{spherical nilpotent elements}
$$
	\calN_{\mathrm{sph}} = \{e \in \calN \; | \; \height(e) \leq 3\}.
$$
The name that we used for these elements is related to  the following geo\-metrical characterization due to Panyushev (see \cite{Pa1} and \cite{Pa2}): for $e \in \calN$, $\height(e) \leq 3$ if and only if $Ge$ is a \textit{spherical variety}, namely every Borel subgroup of $G$ has a dense open orbit in $Ge$. More explicitly, the spherical nilpotent orbits (that is, the orbits of the spherical nilpotent elements) can be characterized as those admitting a representative which is a sum of root vectors corresponding to pairwise orthogonal simple roots (see \cite{Pa2}).

 As noticed in \cite{Pa1}, the height of a nilpotent element is always even when $G$ is a sphecial linear or a symplectic group: thus in these cases an element $e \in \calN$ is spherical if and only if $\height(e) \leq 2$, if and only if $e^2 = 0$.
 
As follows from a general theorem independently proved by Brion \cite{brion0} and Vinberg  \cite{vinberg}, the fact that $Ge$ is a spherical variety also implies that every Borel subgroup $B \subset G$ acts on the closure $\overline{Ge}$ with finitely many orbits. Since $\calN_{\mathrm{sph}}$ is itself the closure of a spherical nilpotent orbit, it follows that every Borel subgroup of $G$ acts on $\calN_{\mathrm{sph}}$ with finitely many orbits: it is therefore natural to study the $B$-orbits therein, together with the associated partial order induced by the inclusion of closures. This is the main object of the present paper, where we will restrict our attention to the height 2 nilpotent locus
$$
	\calN_2 = \{e \in \calN \; | \; \height(e) \leq 2\}.
$$

In order to be more precise in the description of the main results of the paper, we introduce some further notation. Fix  a Borel subgroup $B$ of $G$ and  a maximal torus $T$ of $B$. Denote by $\Phi$ the set of roots of $G$ with respect defined by $T$, let $\grD \subset \Phi$ be the set of simple roots defined by $B$ and let $W = N_G(T)/T$ be the associated Weyl group. Let $\gog = \got \oplus \bigoplus_{\gra \in \Phi} \gog_\gra$ be the root space decomposition of $\gog$, and fix a non-zero element $e_\gra \in \gog_\gra$ for all $\gra \in \Phi$.

Recall that two roots $\gra, \grb \in \Phi$ are called \textit{strongly orthogonal} if neither their sum nor their difference is a root (if $G$ is simply laced, this is equivalent to the fact that $\gra, \grb$ are orthogonal roots). More generally, we say that a set of roots $S \subset \Phi$ is \textit{strongly orthogonal} if its elements are pairwise strongly orthogonal. To any strongly orthogonal subset $S \subset \Phi$ we associate a nilpotent element by setting
$$
	e_S = \sum_{\gra \in S} e_\gra.
$$
Such an element is indeed nilpotent of height at most 4, and if moreover $G$ is a special linear or a symplectic group then $\height(e_S) = 2$ for all nonempty strongly orthogonal subset $S \subset \Phi$ (see Proposition \ref{prop:nilpotent} and Remark \ref{oss:strongly-orth-AC}).

Nilpotent elements of the previous shape are useful in order to study $B$-orbits of nilpotent elements of small height. If indeed $R,S \subset \Phi$ are strongly orthogonal, then we have $B e_S = B e_R$ if and only if $S = R$ (see Proposition \ref{prop:strongly-orth-uniqueness}). If one restricts the attention to the orthogonal subsets arising in $\calN_2$, the situation is even  better: indeed in this case
$$
	\calN_2 = \bigcup_{\height(e_S) \leq 2} B e_S.
$$

In order to study the $B$-orbits in $\calN_2$, it is therefore enough to study the orbits of the ``orthogonal nilpotent elements" of shape $e_S$.

To study the inclusion relations among the closures of the $B$-orbits of the elements $e_S$, we associate to any strongly orthogonal subset $S \subset \Phi$ an involution of the affine Weyl group $\wh W$. If $\wh \Phi$ is the affine root system attached to $\Phi$ and if $\grd \in \wh \Phi$ is the fundamental imaginary root, define
$$
	\wh S = \{\gra - \grd \; | \; \gra \in S\}.
$$

Endow $\wh W$ with the Bruhat order defined by the set of simple roots $\grD \cup \{\grd - \theta\} \subset \wh \Phi$ (where $\theta \in \Phi$ is the highest root defined by $\grD$) and let $\ell : \wh W \rightarrow \mN$ be the associated length function. If $\gra \in \wh \Phi_{\mathrm{re}}$, let $s_\gra \in \wh W$ be the corresponding reflection, and if $S \subset \wh \Phi_{\mathrm{re}}$ is a set of pairwise strongly orthogonal roots define
$\grs_S = \prod_{\gra \in S} s_\gra$.

The following is our main theorem.

\begin{theorem*}[see Corollary \ref{cor:dim-formula} and Theorem \ref{teo:bruhat3}] \label{teo:teorema1}
Let $R,S \subset \Phi$ be strongly orthogonal with $\height(e_R) = \height(e_S) = 2$, then $Be_R \subset \overline{Be_S}$ if and only if $\grs_{\wh R} \leq \grs_{\wh S}$. Moreover, we have
$$
	\dim(Be_S) = \frac{\ell(\grs_{\wh S}) + | S|}{2}.
$$
\end{theorem*}

We point out that, if $R,S \subset \Phi$ are strongly orthogonal and $Be_R \subset \overline{Be_S}$, then the inequality $\grs_{\wh R} \leq \grs_{\wh S}$ holds without further assumptions on the height (see Proposition \ref{prop:bruhat-involuzioni}). On the other hand, without assumptions on the height, different strongly orthogonal subsets are not necessarily separated by the corresponding affine involutions (see Corollary \ref{cor:iniettivita-involuzioni} and \cite[Remark 4.4]{GMMP}).

The present paper stems as a continuation of \cite{GMMP}, where the case of the abelian ideals of the Lie algebra $\gob$ of $B$ is considered. If indeed $\goa$ is such an ideal, then by a result of Panyushev and R\"ohrle \cite{PR} it holds that  $\height(e) \leq 3$ for all $e \in \goa$, so that $\goa \subset \calN_{\mathrm{sph}}$. The set of the $B$-orbits on $\goa$ was considered by Panyushev in \cite{Pa3}, where these orbits are parametrized in terms of strongly orthogonal sets of roots. The affine Weyl group was then brought into the picture in \cite{GMMP}, where a statement analogous to that in Theorem \ref{teo:teorema1} is proved in the case of the abelian ideals of $\gob$. This case will be indeed a main step in the proof of Theorem \ref{teo:teorema1}.

In proving Theorem 1, we will also study the action of $B$ on some well known resolutions of singularities associated to the closure of a nilpotent $G$-orbit in $\calN_2$. To better explain our results in this direction, we briefly recall how these resolutions are defined. This will also explain in which sense the present paper extends the results obtained in \cite{GMMP}.

The height of an element $e \in \calN$ can be nicely expressed in terms of the $\mZ$-grading associated to a characteristic of $Ge$. If indeed $h$ is the semisimple element of an $\mathfrak{sl}_2$-triple containing $e$, then the eigenspace decomposition $\gog = \bigoplus_{i \in \mZ} \gog(i)$ defined by $\ad(h)$ induces a $\mZ$-grading of $\gog$, and the height of $e$ is the largest eigenvalue of $\ad(h)$ on $\gog$. Set $\calO = Ge$ and let $\ol \calO$ be its closure. Up to conjugation we may always assume that $h$ is the dominant characteristic of $\calO$. Let $P$ be the standard parabolic subgroup of $G$ with Lie algebra $\gop = \bigoplus_{i \geq 0} \gog(i)$, and set $\goa = \bigoplus_{i \geq 2} \gog(i)$: then $\goa$ is an ideal of $\gop$ contained in its nilradical $\gop^u$. It is well known that $G \goa = \ol \calO$, and that the contraction
$$
	G \times_P \goa \longrightarrow \overline \calO, \qquad [g,x]  \longmapsto gx
$$
is a resolution of singularieties. Moreover, $\calO$ is a spherical nilpotent orbit if and only if $\goa$ is an abelian ideal of $\gob$, and $\calO \subset \calN_2$ if and only if the unipotent radical $P^u$ acts trivially on $\goa$.

Suppose now that $\calO \subset \calN_2$ and denote $\wt \calO = G \times_P \goa$. In this case the orbit structure of $\wt \calO$ largely reduces to that of the abelian ideal $\goa \subset \gob$. In particular, the $B$-orbits in $\wt \calO$ are completely determined by their images inside the flag variety $G/P$ and inside the closure $\ol \calO$. More precisely, let $\Psi \subset \Phi^+$ be the set of roots occurring in $\goa$, let $\Ort(\Psi)$ be the family of the orthogonal subsets of $\Psi$, and let $W^P$ be the set of the minimal length coset representatives of $W/W_P$. Then we have bijections
$$
	B \backslash \wt \calO \; \longleftrightarrow \; B\backslash G /P \times B \backslash \goa \; \longleftrightarrow \; W^P \times \Ort(\Psi),
$$
sending $(w,S) \in W^P  \, \times \, \Ort(\Psi)$ to the $B$-orbit $B[w,e_S] \subset \wt \calO$.

When the involved orthogonal subsets arise from $\overline \calO$, we will deduce Theorem \ref{teo:teorema1} from the following description of the orbit structure of $\wt \calO$.

\begin{theorem*}[{see Theorem \ref{teo:bruhat1}}] \label{teo:teorema2}
Suppose that $\calO \subset \calN_2$ and let $(v,R)$, $(w,S)$ be in $W^P \times \Ort(\Psi)$. Then  $B[v,e_R]  \subset \overline{ B[w,e_S]}$ if and only if $v \leq w$ and $\grs_{v(\wh R)} \leq  \grs_{w(\wh S)}$.
\end{theorem*}

The proofs of Theorem \ref{teo:teorema1} and Theorem \ref{teo:teorema2} are by induction. At the basis of the inductions we will find the results obtained in \cite{GMMP} in the case of the abelian ideals of $\gob$. The inductive step is based on the action of the minimal parabolic subgroups of $G$: adapting the approach of Richardson and Springer \cite{RS} to our context (as already done in \cite{GMMP}), we will be able to control the orbits of the minimal parabolic subgroups on $\calN_2$ in terms of the affine involutions associated to the orthogonal subsets whose corresponding nilpotent elements are in $\calN_2$.

As no assumption on the height was involved in \cite{GMMP}, it seems reasonable that both Theorem \ref{teo:teorema1} and  Theorem \ref{teo:teorema2} might hold in the general spherical case. Beyond the technical problems arising in this more general context, the main obstruction lies in the fact that the orbits of  shape $Be_S$ do not cover the full spherical locus $\calN_{\mathrm{sph}}$.

When $G$ is the special linear group $\mathrm{SL}_n(\mk)$, the problem addressed in Theorem \ref{teo:teorema1} was considered by Melnikov in the papers \cite{melnikov}, \cite{melnikov2}, \cite{melnikov3} and by Boos and Reineke in \cite{boos1}. Motivated by the study of the orbital varieties (that is, the irreducible components of the intersection of a nilpotent orbit with $\gob$), Melnikov studied the action of the Borel subgroup $B$ of upper triangular matrices on $\calN_2 \cap \gob$. More precisley, a parametrization of the $B$-orbits in $\calN_2 \cap \gob$ in terms of the involutions in the symmetric group was given  in \cite{melnikov}, whereas the partial order among these orbits was described in \cite{melnikov2} by introducing a new partial order among the involutions substantially different from their usual Bruhat order. The problem was then translated in the combinatorial language of link patterns in \cite{melnikov3}, and generalized by Boos and Reineke in \cite{boos1} in order to give a complete description of the $B$-orbits in $\calN_2$ and of their partial order. The parametrization of the $B$-orbits in terms of link patterns was later generalized to all classical groups by Boos, Cerulli Irelli and Esposito in \cite{boos2}, where  the action of the Borel subgroups of $G$ on the subvariety of $\calN_2$ defined by the equation $e^2 = 0$ is considered.

The singularities of the closures of the $B$-orbits in $\calN_2$ have been investigated by Bender and Perrin in \cite{bender}. When $G$ is a classical group, there are also given explicit descriptions of the $B$-orbits in $\calN_2$ and of their partial order in a case-by-case language.

Other parametrizations of the $B$-orbits in $\calN_2$ have been given by Chaput, Fresse and Gobet in the recent preprint \cite{CFG}. In the case of $\mathrm{SL}_n(\mk)$, there is provided yet another description of the partial order among the $B$-orbits.

Finally, we mention the work of Ignatyev \cite{ignatyev1}, \cite{ignatyev2}, \cite{ignatyev3}. If $G$ is a classical group, in these papers the author attach to any involution in $W$ a coadjoint $B$-orbit in the dual Lie algebra $\gou^*$ (where $\gou$ denotes the nilradical of $\gob$), and the partial order among these orbits is studied in terms of the Bruhat order of the corresponding involutions.

We now describe the structure of the paper. In Section 1 we give preliminaries and set up  notation. In Section 2 we study in full generality the $B$-orbits in $\calN$ associated to a strongly orthogonal sets of roots. In Section 3 we restrict to the case of a nilpotent orbit $Ge \subset \calN_2$, and we study the action of $B$ on its closure $\overline{Ge}$ and on its rational resolution $\wt{Ge} \rightarrow \overline{Ge}$. Finally, in Section 4  we prove Theorem \ref{teo:teorema2}, and in Section 5 we prove Theorem \ref{teo:teorema1}.

\textit{Acknowledgements.} We thank A. Maffei for useful discussions on the subject, and the anonymous referee for his/her reading and comments.

\section{Notation and preliminaries.}

In this section we clarify the notation that will be used throughout the paper, and will expand some of the preliminaries already outlined in the introduction.

Throughout the paper $G$ will be an almost simple group over an algebraically closed field $\mk$ of characteristic zero, with a fixed Borel subgroup $B$ and fixed maximal torus $T \subset B$. We will denote by $W =  N_G(T)/T$ the Weyl group of $G$ respect to $T$ and by $\Phi$ the set of roots of $G$ defined by $T$, whereas $\Phi^+ \subset \Phi$ and $\grD \subset \Phi^+$ will denote respectively the set of positive and of simple roots of $\Phi$ defined by $B$. We also set  $\Phi^- = \Phi \setminus \Phi^+$. We will regard $W$ endowed with the Bruhat order $\leq$ and with the length function $\ell : W \rightarrow \mN$ defined by $\grD$. The longest element of $W$ will be denoted $w_0$.

If $H \subset G$ is a closed subgroup, we will denote by $H^u$ the unipotent radical of $H$ and by $H^\circ$ the identity component. The Lie algebra of $H$ will be denoted with the corresponding gothic letter $\goh$. If $g \in G$, then we put  $g.H = gHg^{-1}$.

We regard the Lie algebra $\got$ of $T$ as a Euclidean vector space endowed with the invariant scalar product induced by the Killing form $\kappa$ on $\gog$. This induces a $W$-invariant scalar product $(.,.)$ on $\got^* \supset \Phi$. The set of long roots in $\Phi$ will be denoted by $\Phi_\ell$, and the set of short roots by $\Phi_s$. When all the root have the same length, every root will be regarded as long. If $\gra \in \Phi$, the coroot of $\gra$ will be denoted by $\gra^\vee$, and the pairing between roots and coroots will be denoted by $\langle .,. \rangle$. 

If $\gra \in \Phi$, we will denote by $\gog_\gra$ the corresponding root space and by $U_\gra = \exp(\gog_\gra)$ the corresponding root subgroup. For all $\gra \in \Phi$, we fix  nonzero elements $e_\gra \in \gog_\gra$, $h_\gra \in \got$ and $f_\gra \in \gog_{-\gra}$ in such a way that $\{e_\gra, h_\gra, f_\gra \}$ is an $\mathfrak{sl}_2$-triple. The reflection of $W$ defined by $\gra$ will be denoted by $s_\gra$, it can be represented in $N_G(T)$ by exponentiating the elements in the corresponding triple as follows
$$
	s_\gra = \exp(-f_\gra) \exp(e_\gra) \exp(-f_\gra).
$$ 

If $w \in W$, we will define its \textit{set of inversions} as
$\Phi^+(w) = \{\grb \in \Phi^+ \; | \; w(\grb) \in \Phi^-\}$.
We will regard $\Phi$ as a partially ordered set with the dominance order $\leq$, defined by $\gra \leq \grb$ if and only if $\grb - \gra \in \mN \grD$. The highest root of $\Phi$ will be denoted by $\theta$.

We record in the following lemma some consequences of the subword and of the lifting properties of Coxeter groups \cite{BB} that will be useful in the paper.

\begin{lemma} \label{lemma:par1}
Let $\gra\in \Delta$ and let $v,w \in W$ be such that $v < w$. The following hold:
\begin{itemize}
 \item[i)]  If $s_\gra v > v$ and $s_\gra w> w$, then $s_\gra v < s_\gra w$.
 \item[ii)]  If $s_\gra v < v$ and $s_\gra w < w$, then $s_\gra v < s_\gra w$. 
 \item[iii)]  If $s_\gra v > v$ and $s_\gra w < w$, then $s_\gra v \leq w$ and $v \leq s_\gra w$.
\end{itemize}
\end{lemma}

Let $P$ be a standard parabolic subgroup of $G$ with Levi decompostion $P = LP^u$. We denote by $P^-$ the opposite parabolic subgroup of $P$ and by $B_L = B \cap L$ the Borel subgroup of $L$ defined by $B$. Moreover $W_L$ denotes the Weyl group of $L$ (regarded as a subgroup of $W$), and $W^P$ denotes the set of minimal length coset representatives of $W/W_L$. Recall that any $w\in W$ has a unique decomposition of the shape $w=u v$ with $u \in W^P$ and $v \in W_L$, and $\ell(w)=\ell(u)+\ell(v)$: the element $u$ is the element of minimal length in $wW_L$ and will be denoted by $(w)^P$, whereas we will denote $v = (w)_P$. The longest element of $W_L$ will be denoted by $w_L$.

Let $\calO \subset \calN$ be a nilpotent orbit. Recall that a semisimple element $h \in \gog$ is said to be a \textit{characteristic} for $\calO$ if there exists $e \in \calO$ such that $e$ is the nilpositive element of an $\mathfrak{sl}_2$-triple with semisimple element $h$. If $h$ is a characteristic for $\calO$, it is well known that $Gh \cap \got$ contains a unique dominant element, namely an element $h_+ \in \got$ such that $\gra(h_+) \geq 0$ for all $\gra \in \grD$. The \textit{weighted Dynkin diagram} of $\calO$ is the Dynkin diagram of $G$ labelled with the non-negative integers $\gra(h_+)$ for $\gra \in \grD$; it uniquely determines the $G$-orbit of $e$ (see \cite{CMcG} for more details). 

If $\{e,h,f\}$ is an $\mathfrak{sl}_2$-triple in $\gog$, we set
$$
	\gog(i,h) = \{x \in \gog \; | \; [h,x] = i x\}.
$$
the eigenspace of weight $i \in \mZ$ defined by $\ad(h)$. As already recalled, the height of $e$ is the largest eigenvalue of $\ad(h)$ on $\gog$ (see \cite{Pa2}). When $h \in \got$, as will always be in our setting, the eigenspaces $\gog(i,h)$ are all $\got$-stable, and we set 
$$\Phi(i,h) = \{ \gra \in \Phi \; | \; \gog_\gra \subset \gog(i,h)\}.$$
When the semisimple element $h$ is clear from the context, we will denote $\gog(i,h)$ and $\Phi(i,h)$ simply by $\gog(i)$ and $\Phi(i)$.

Let $\widehat \gog =\gog[z,z^{-1}]\oplus \mk C \oplus \mk d$ be the affinization of $\gog$, let $\wh \Phi = \wh\Phi_{\mathrm{re}}\sqcup\pm\mathbb N\grd$ be the cor\-res\-pon\-ding affine root system, with real roots 
$\wh\Phi_{\mathrm{re}}=\Phi \pm \mZ \grd$ and fundamental imaginary root $\delta$. Let $\wh W$ be the  Weyl group of  $\wh\Phi$. We denote by $\leq$ the Bruhat order on $\wh W$ defined by the set of simple roots  $\wh \grD = \grD \cup \{\grd-\theta\} \subset \wh \Phi$, and by $\ell : \wh W \rightarrow \mN$ the corresponding length function. Given  $\gra \in \wh \Phi_{\mathrm{re}}$ we denote by $s_\gra$ the corresponding reflection in $\wh W$, and if $S \subset \wh \Phi_{\mathrm{re}}$ is a subset of pairwise orthogonal roots we set $\grs_S = \prod_{\gra \in S}s_\gra$.

We denote by $\widehat G$ the Kac-Moody group associated to $\widehat\gog$, and by $\widehat T=T \times \mk^*_C \times \mk^*_d$
the maximal torus of $\widehat G$ containing $T$ whose Lie algebra contains $C$ and $d$. Given $\gra\in \widehat \Phi_{\mathrm{re}}$, let $\gog_\gra \subset \widehat \gog$ be the corresponding root space. In particular if $\gra \in \Phi$ and $n \in \mZ$, then we have $\gog_{\gra + n\grd} = z^n\gog_\gra$. Moreover we can choose the root vectors in such a way that the reflection $s_{\gra+n\grd}\in \wh W$ is represented in $\wh G$ as
$$ s_{\gra+n\grd} = \exp(-z^{-n} f_\gra)\, \exp(z^n e_\gra) \,\exp(-z^{-n} f_\gra). $$

We now recall some properties of the set $\wh \calI$ of the involutions in $\wh W$, regarded as a poset with the Bruhat order. The definition and the statements are adapted from the work of Richardson and Springer \cite{RS}.

If $\grs \in \wh \calI$, the \emph{length} of $\grs$ (regarded as an involution) is 
$$
L(\grs)=\frac{\ell(\grs)+\rk (\mathrm{id}-\grs)}{2}.
$$
Notice that if $\grs = \grs_S$ for some orthogonal set $S \subset \wh \Phi_{\mathrm{Re}}$, as it will always be in our setting, then $\rk (\mathrm{id}-\grs) = \card(S)$.

Given $\gra \in \wh \grD$ and $\grs \in \wh \calI$, define
$$
s_\gra \circ \grs=
\begin{cases}
 s_\gra \grs &\text{if } s_\gra \grs = \grs s_\gra \\
 s_\gra \grs s_\gra &\text{if } s_\gra \grs \neq \grs s_\gra
\end{cases}
$$

The simple root $\gra$ is said to be a \textit{descent} for $\grs$ if $\grs(\gra) < 0$. If moreover $\grs(\gra) = -\gra$, then $\gra$ is called a \textit{real descent}, otherwise it is called a \textit{complex descent}.

In the following lemma we record some connections between the descents of an involution and its length as an involution (see e.g. \cite[Lemma 2.6]{GMMP}).

\begin{lemma}\label{lemma:discese-CR}
Let $\grs \in \wh \calI$ and $\gra\in \wh \grD$, then the following statements are equivalent: 
\begin{itemize}
        \item[i)] $\gra$ is a descent for $\grs$; 
        \item[ii)] $s_\gra\grs<\grs$;
        \item[iii)] $\grs\,s_\gra<\grs$;
 		\item[iv)] $s_\gra\circ \grs <\grs$; 
        \item[v)] $L(s_\gra\circ \grs)=L(\grs)-1$.
\end{itemize}

If moreover $\alpha$ is a descent for $\grs$, then it is real if and only if $s_\gra\,\grs=\grs \, s_\gra$, and it is complex if and only if
$s_\gra\grs\,s_\gra<s_\gra \grs$ and $s_\gra\grs\,s_\gra<\grs \,s_\gra$. 
\end{lemma}

In particular, if $\grs \in \wh \calI$ and $\gra \in \wh \grD$, then $\gra$ is a complex descent for $\grs$ if and only if $\ell(s_\gra \circ \grs) = \ell(\grs) -2$, whereas it is a real descent if and only if $\ell(s_\gra \circ \grs) = \ell(\grs) -1$.

Finally in the following lemma we record a statement such as Lemma \ref{lemma:par1} in the context of $\wh \calI$ (see e.g. \cite[Lemma 4.7]{GMMP}).

\begin{lemma} \label{lemma:par2}
Let $\gra\in \widehat \Delta$ and let $\grs, \tau \in \wh \calI$ be such that $\grs < \tau$. The following hold:
\begin{itemize}
 \item[i)]  If $s_\gra \circ \grs > \grs$ and $s_\gra \circ \tau > \tau$, then $s_\gra \circ \grs < s_\gra \circ \tau$.
 \item[ii)]  If $s_\gra \circ \grs < \grs$ and $s_\gra \circ \tau < \tau$, then $s_\gra \circ \grs < s_\gra \circ \tau$. 
 \item[iii)]  If $s_\gra \circ \grs > \grs$ and $s_\gra \circ \tau < \tau$, then $s_\gra \circ \grs \leq \tau$ and $\grs \leq s_\gra \circ \tau$.
\end{itemize}
\end{lemma}

\section{Nilpotent elements associated to strongly orthogonal\\ sets of roots}

Given $S \subset \Phi$ a strongly orthogonal subset, set
$$
	e_S = \sum_{\gra \in S} e_\gra, \qquad 
	h_S = \sum_{\gra \in S} h_\gra, \qquad 
	f_S = \sum_{\gra \in S} f_\gra
$$

\begin{proposition}	[see {\cite[Lemma 6.1]{GMP}}]	\label{prop:nilpotent}
Let $S \subset \Phi$ be a strongly orthogonal subset, then $e_S \in \calN$, and $\height(e_S) \leq 4$.
\end{proposition}

\begin{proof}
Suppose that $S$ is nonempty. Since the roots in $S$ are pairwise strongly orthogonal, it follows that $\{e_S, h_S, f_S\}$ is an $\mathfrak{sl}_2$-triple, with semisimple element $h_S$. Thus $e_S$ and $f_S$ are nilpotent.

Let $\grg \in \Phi$ be such that $\grg (h_S) = \height(e_S)$, and denote $S_0 = \{\gra \in S \; | \; \langle \grg, \gra^\vee \rangle > 0\}$. Then
\begin{equation}	\label{eq:height}
	\height(e_S) = \grg(h_S) = \sum_{\gra \in S} \langle \grg,\gra^\vee \rangle \leq \sum_{\gra \in S_0} \langle \grg, \gra^\vee \rangle.
\end{equation}

Set $\wt S_0 = S_0 \cup \{-\grg\}$, then by \cite[Lemma 5.2]{GMP} the matrix $(\langle \gra, \grb^\vee \rangle )_{\gra, \grb \in \wt S_0}$ is a generalized Cartan matrix of finite or affine type. Notice that every node of the corresponding Dynkin diagram is connected with the node associated to $\grg$, and that by construction the degree of $\grg$ as a vertex of this Dynkin diagram is $\sum_{\gra \in S_0} \langle \grg, \gra^\vee \rangle$. Thus the first claim follows by the classification of the generalized Cartan matrices of finite or affine type, as the maximum number of edges connected to a single node in a Dynkin diagram is at most 4.
\end{proof}

\begin{remark}	\label{oss:strongly-orth-AC}
Suppose that $\Phi$ is of type $A$ or $C$, then $\height(e_S) \leq 2$ for all strongly orthogonal subset $S \subset \Phi$.

To see this, we use the usual $\varepsilon$-notation to describe the roots in these cases: if $\gre_1, \ldots, \gre_n$ is an orthonormal basis in $\mathbb R^n$ with scalar product $(.,.)$. Then
\begin{align*}
	& \Phi = \{\gre_i - \gre_j \; | \; 1 \leq i \neq j \leq n \} \qquad & \text{if $\Phi$ is of type $A_{n-1}$,} \\
	& \Phi = \{\gre_i \pm \gre_j \; | \; 1 \leq i \neq j \leq n \} \cup \{2 \varepsilon_i \; | \; 1 \leq i \leq n\} \qquad & \text{if $\Phi$ is of type $C_n$.}
\end{align*}

If $\grg \in \Phi$ and $\grb_1, \ldots, \grb_k \in \Phi$ are pairwise strongly orthogonal roots such that
$$
	(\grb_i , \grg) > 0 \qquad \forall i = 1, \ldots, k
$$
from the description of $\Phi$ we immediately see that $k \leq 2$. If moreover $k=2$ and $\{\grg, \grb_1, \grb_2\}$ contains roots of different lengths, then $\grg$ must be a short root.

Let $S \subset \Phi$ be a nonempty strongly orthogonal subset. As in the previous proof, let $\grg \in \Phi$ be such that $\grg (h_S) = \height(e_S)$, and denote $S_0 = \{\gra \in S \; | \; \langle \grg, \gra^\vee \rangle > 0\}$. Then it follows that $|S_0| \leq 2$, and moreover we either have $S_0 = \{\grg\}$ or $\langle \grg, \grb^\vee \rangle = 1$ for all $\grb \in S_0$. Thus $\sum_{\gra \in S_0} \langle \grg, \gra^\vee \rangle = 2$, and the claim follows from the inequality \eqref{eq:height}.
\end{remark}

We now want to show that the map $S \mapsto Be_S$ gives an injection
$$
	\{\text{strongly orthogonal subsets of } \Phi\}	\longrightarrow B \backslash \calN.
$$

We now show a fundamental relation between the orbit $Be_S$ and the corresponding affine involution $\grs_{\wh S}$.

\begin{proposition} \label{prop:bruhat-involuzioni}
Let $R,S \subset \Phi$ be strongly orthogonal and suppose that $Be_R \subset \overline{B e_S}$. Then $\grs_{\wh R} \leq \grs_{\wh S}$.
\end{proposition}

\begin{proof}
Notice that the exponential induces a $G$-equivariant morphism
$$
	\exp : \, \gog \otimes z^{-1} \longrightarrow \wh G, \qquad x \otimes z^{-1} \longmapsto \exp(z^{-1}x)
$$

Consider the shifted vectors $e_{\wh S} = e_S \otimes z^{-1}$ and $f_{\wh S} = f_S \otimes z$. Since $\wh S$ is strongly orthogonal, we have
$$	\grs_{\wh S} = \exp(-f_{\wh S}) \exp(e_{\wh S}) \exp(-f_{\wh S}).
$$

Let $\wh B \subset \wh G$ be the Iwahori subgroup defined by $\wh \grD$. Since $\exp(-f_{\wh S}) \in \wh B$, we have
$$\exp(B e_{\wh S}) \subset  \wh B \exp(e_{\wh S}) \wh B = \wh B \exp(-f_{\wh S}) \exp(e_{\wh S}) \exp(-f_{\wh S}) \wh B.
$$
In particular, it follows that $\exp(B e_{\wh S}) \subset \wh B \grs_{\wh S} \wh B$.

Since $\gog \simeq \gog \otimes z^{-1}$ as $G$-modules, the assumption implies that $e_{\wh R} \in \overline{B e_{\wh S}}$. Thus by the previous remark $\exp(e_{\wh R})$ is in the closure of $\wh B \grs_{\wh S} \wh B$. It follows that $\grs_{\wh R}$ is also in the closure of $\wh B \grs_{\wh S} \wh B$, namely $\grs_{\wh R} \leq \grs_{\wh S}$.
\end{proof}

\begin{definition}
We will say that a $B$-orbit $\calO \subset \calN$ is \textit{strongly orthogonal} if  $\calO = B e_S$ for some strongly orthogonal $S \subset \Phi$, in which case we set $\calO = \calO_S$.
\end{definition}

\begin{remark} \
\begin{itemize}
	\item[i)] Notice that the $B$-orbit $\calO_S$ only depends on $S$, and not on the normalization of the root vectors $e_\gra$ in the definition of $e_S$. Since $S$ is orthogonal, we have indeed
$$
	Te_S = \{\sum_{\gra \in S} \xi_\gra e_\gra \; | \; \xi_\gra \in \mk^*\}. 
$$
	\item[ii)] By Proposition \ref{prop:B-orbits-in-X} every strongly orthogonal $B$-orbit has height at most 4, on the other hand we will see in Corollary \ref{cor:height2-is-orthogonal} that every nilpotent $B$-orbit of height 2 is strongly orthogonal. In particular, by Remark \ref{oss:strongly-orth-AC} we see that if $G$ is of type $AC$ then a nonzero $B$-orbit in $\calN$ is strongly orthogonal if and only if it has height 2. If instead $G$ is not of type $AC$, then there do exist strongly orthogonal $B$-orbits of any height between 2 and 4.
	\end{itemize}
\end{remark}

Notice that by  Proposition \ref{prop:bruhat-involuzioni} the assignment $\calO_S \mapsto \grs_{\wh S}$ gives a well defined map
$$
	\{\text{strongly orthogonal $B$-orbits in $\calN$}\} \longrightarrow  \mathcal{I}
$$

In general, it may happen that different strongly orthogonal subsets give rise to the same involution in $\wh W$ (see \cite[Remark 4.4]{GMMP}). On the other hand, different strongly orthogonal subsets always give rise to different nilpotent $B$-orbits. This will be proved as a consequence of a general geometric property satisfied by $Te_S$ as a $T$-orbit inside $Be_S$. We first need a couple of combinatorial lemmas.

\begin{lemma}	\label{lemma: roots-in-lattice}
Let $S \subset \Phi$ be strongly orthogonal, then $\mZ S \cap \Phi = S \cup (-S)$.
\end{lemma}

\begin{proof}
The claim is easily checked directly if $\Phi$ is of type $G_2$, so we can assume  that this is not the case. Let $\gra \in {\Phi}$ and suppose that $\gra = \sum a_i \grb_i$ with $\grb_i \in S$ and $a_i$ nonzero integers. By the orthogonality of $S$, we have $||\gra||^2 = \sum a_i^2 ||\grb_i||^2$.Thus the ratio $\tfrac{||\gra||^2}{||\grb_1||^2}$ is either $1$ or $2$. Suppose that $||\gra||^2 = 2||\grb_1||^2$: then it must be $\gra = \pm \grb_1 \pm \grb_2$, which is absurd since $\grb_1$ and $\grb_2$ are strongly orthogonal. Therefore $||\gra||^2 = ||\grb_1||^2$, and it follows that $\gra = \pm \grb_1$.
\end{proof}

\begin{lemma}		\label{lemma:01}
Let $S\subset\Phi$ be orthogonal, and let $S_0 \subset S \cup (-S)$ be such that $\grs_{\wh S_0} = \grs_{\wh S}$. Then $S_0=S$.
\end{lemma}

\begin{proof}
Write  $S=\{\beta_1,\ldots\beta_k\}$ and $S_0=\{\beta_1, \ldots, \beta_h, -\beta_{h+1}, \ldots, -\beta_{h+s},\ldots, -\beta_r\}$ with $h\leq h+s \leq r \leq k$ and with $\beta_{h+j} \in \Phi^+$ if and only if $1\leq j\leq s$. By the assumption we have
$$\prod_{j=h+1}^k s_{\delta-\beta_j}=\prod_{j=h+1}^r s_{\delta+\beta_j}=\prod_{j=h+1}^{h+s} s_{\delta+\beta_j}\prod_{j=h+s+1}^r s_{\delta+\beta_j}$$

By looking at the affine action of both sides on $\got$, we show that this equality cannot hold unless $h=k$. Let $s_{\gra,k} \in \mathrm{Aff}(\got)$ be the affine reflection defined by $s_{\gra,k}(x)=x- (\gra(x) -k)\gra^\vee$, then we have an isomorphism of $\widehat W$ onto the subgroup of $\mathrm{Aff}(\got)$ generated by the $s_{\gra,k}$ mapping the reflection $s_{k \delta \pm\gra}$ (for $\alpha\in\Phi^+$) to $s_{\gra,\mp k}$. Assume  $h<k$, and  choose $y\in \got$ such that $\beta_j(y) = -1$ for $h < j \leq r$ and $\beta_j(y)=1$ for $j>r$. Then $y$ is fixed by $\prod_{j=h+1}^r s_{\delta+\beta_j}$, whereas
$$
\Big( \prod_{j=h+1}^k s_{\delta-\beta_j} \Big) (y) = \Big( \prod_{j=h+1}^r s_{\delta-\beta_j} \Big) (y) = y+2\sum_{j=h+1}^r\beta^\vee_j.
$$
Therefore $\sum_{j=h+1}^r\beta^\vee_j = 0$, which is absurd because the roots are orthogonal.
\end{proof}

If $Z$ is an algebraic variety acted by a connected solvable algebraic group $K$, we can associate to $Z$ a sublattice of the character group $\calX(K)$, called the weight lattice of $Z$ and defined as follows:
$$\calX_K(Z) = \{ \text{weights of rational $K$-eigenfunctions $f\in\mk(Z)$}\}.$$
Recall that a base point $e_0 \in Be_S$ is called $T$-\textit{standard} if $\calX_B(Be_S) = \calX_T(Be_0)$, equivalently if $\Stab_T(e_0)$ is a maximal diagonalizable subgroup of $\Stab_B(e_0)$ (see e.g. \cite[Section 4.1]{GM}). 

\begin{proposition} 	\label{prop:strongly-orth-uniqueness}
Let $S \subset \Phi$ be strongly orthogonal. Then $e_0 \in B e_S$ is a $T$-standard base point if and only if $e_0 \in T e_S$.
In particular, the map
$$
	\{\text{strongly orthogonal subsets of } \Phi\}	\longrightarrow B \backslash \calN
$$
is injective.
\end{proposition}

\begin{proof}
We adapt the proof of \cite[Proposition 4.5]{GM} to this slightly different context. We keep the notation therein. Let $e_0 \in B e_S$ be a $T$-standard base point and denote $S_0 = \supp(e_0)$,
then $\calX_B(B e_S) = \calX_T(Te_0) = \mZ S_0$ (see Lemma 4.4 in \cite{GM}). The restriction of functions gives an inclusion $\calX_B(B e_S) \subset \calX_T(Te_S) = \mZ S$, thus by Lemma \ref{lemma: roots-in-lattice} we get $S_0 \subset S \cup (-S)$. On the other hand, since $B e_{S_0} = B e_S$, by Proposition \ref{prop:bruhat-involuzioni} we have $\sigma_{\wh S_0}=\sigma_{\wh S}$. Therefore $S_0 = S$ thanks to Lemma \ref{lemma:01}.
\end{proof}

As we already noticed, it may happen that different strongly orthogonal subsets give rise to the same involution in $\wh \calI$. The situation gets better if we consider strongly orthogonal subsets of height at most 3:  in this case the strongly orthogonal subset (hence the strongly orthogonal nilpotent $B$-orbits) are separated by the associated involutions. First we need a combinatorial lemma.

\begin{lemma}	[{\cite[Proposition 4.1]{GMMP}}] \label{lemma:radici-attive-reali}
Let $S \subset \Phi$ be strongly orthogonal with $\height(e_S) \leq 3$ and let $\gra \in \widehat \Phi$ be such that $\grs_{\wh S}(\gra) = - \gra$.
Then $\gra=\tfrac{1}{2} (\pm\grb\pm\grb')$ with $\grb,\grb'\in \wh S$.
\end{lemma}

\begin{proof}
Since $\grs_{\wh S}(\gra) = -\gra$, we have
\begin{equation}	\label{eq:2alpha}
2\gra = \sum_{\grb\in S}\langle \gra,\grb^\vee\rangle (\grb - \grd).
\end{equation}
Thus the claim follows if we show that $\sum_{\beta\in S}|\langle \gra,\grb^\vee\rangle|= 2$. By the assumption this sum is clearly positive, hence it suffices to  show that it is an even number smaller than 3.

By the assumption on the height of $e_S$ we have $\langle \eta, h_S \rangle \leq 3$ for all $\eta\in\Phi$, thus for all $\eta \in \wh \Phi_{\mathrm{Re}}$. Set $S'=\{\grb\in S\mid \langle \gra,\grb^\vee\rangle< 0\}$, taking $\eta = \grs_{\wh S'}(\gra)$ we get 
$$
	\sum_{\beta\in S}|\langle \gra,\grb^\vee\rangle| = \sum_{\grb\in S}\langle \grs_{\wh S'}(\gra),\grb^\vee\rangle = \langle \grs_{S'}(\alpha), h_S \rangle \le 3.
$$
On the other hand, by \eqref{eq:2alpha} we have
$$
	\grs_{\wh S'}(\alpha) = \gra - \sum_{\grb \in S'} \langle \gra,\grb^\vee\rangle (\grb -\grd) =  \frac{1}{2}\sum_{\grb\in S}| \langle \gra,\grb^\vee\rangle | (\grb - \grd).
$$
Therefore $\sum_{\beta\in S}|\langle \gra,\grb^\vee\rangle| = - 2 \langle \grs_{\wh S'}(\gra), d \rangle < 3$ is a positive even number.
\end{proof}

The following proposition is essentially taken from \cite{GMMP}(see Proposition 4.3 therein).

\begin{proposition} \label{prop:iniettivita-involuzioni}
Let $S \subset \Phi$ be strongly orthogonal with $\height(e_S) \leq 3$ and let $R\subset \Phi$ be an orthogonal subset such that $\grs_{\wh R} = \grs_{\wh S}$, then $R = S$.
\end{proposition}

\begin{proof}
If $\gra \in \wh R$, then 
$\grs_{\wh R}(\gra) = \grs_{\wh S}(\gra) = -\gra$, hence $2\gra = \sum_{\grb\in S}\langle \gra,\grb^\vee\rangle (\grb-\grd)$ and
$\gra\in \Span_\mQ \wh S$. Therefore $\wh R \subset \Span_\mQ \wh S$, and switching the role of $S$ and 
$R$ we obtain $ \Span_\mQ \wh R = \Span_\mQ \wh S$. Assume now that $R \ne S$ and set $S=\{\grb_1,\ldots,\grb_k\}$. 
Let $\gra\in \wh R \smallsetminus \wh S$, then by Lemma \ref{lemma:radici-attive-reali} we can write $\gra = \tfrac 12 (\grb+\grb')$ with $\grb,\grb'$ in $\wh S$. Without loss of generality we can assume that $\gra=\tfrac 12 (\grb_1+\grb_2)$. 
Let $\gra'\in \wh R \smallsetminus \{\gra\}$, and write $\gra'=\sum a_i\grb_i$. Since $\gra$ and $\gra'$ are orthogonal, it follows that $a_1\Vert \grb_1\Vert^2 +a_2\Vert\grb_2\Vert^2=0$, 
hence $a_1=-\frac{ \Vert\grb_2\Vert^2}{ \Vert \grb_1\Vert^2} a_2$. On the other hand, by Lemma \ref{lemma:radici-attive-reali} we can write $\gra'=\tfrac 12(\grb_i+\grb_j)$ for some $\grb_i, \grb_j \in \wh S$, and since $\gra' \neq \gra$ we get $a_1=a_2=0$. It follows that $\wh R \smallsetminus \{\gra\}\subset \Span_\mQ(\wh S\smallsetminus\{\grb_1,\grb_2\})$, against the fact that $S$ and 
$R$ span the same space.
\end{proof}

Combining Propositions \ref{prop:bruhat-involuzioni} and \ref{prop:iniettivita-involuzioni} we get the following.

\begin{corollary} \label{cor:iniettivita-involuzioni}
The map
$$
	\{S \subset \Phi \; | \; S \text{ is strongly orthogonal and $\height(e_S) \leq 3$} \}	\longrightarrow  \wh \calI
$$
is injective.
\end{corollary}

\begin{remark}
Notice that for $\gra\in \mathfrak t^*$ it holds
$$ \grs_{\wh S}(\gra)= \grs_{S}(\gra) +  \gra(h_S) \grd.$$
In particular, the height of $e_S$ is the greatest coefficient of $\grd$ among the roots of shape $\grs_{\wh S}(\gra)$ with $\gra\in\Phi$, and it only depends on $\sigma_{\wh S}$.
\end{remark}

In the following, it will be particularly relevant the case where the strongly orthogonal subset is contained in the set of roots associated to an abelian ideal of $\gob$. Notice that in this case every orthogonal subset is necessarily strongly orthogonal. The following property was noticed in \cite[Theorem 6.5]{GMMP} in the case of an arbitrary abelian ideal of $\gob$. Since no proof was given there, we provide an easy argument suggested by A. Maffei which applies to the context that we will be interested in.

\begin{proposition} \label{prop:bruhat-involuzioni2}
Let $h$ be the dominant characteristic of a nilpotent orbit of height 2, let $\goa = \gog(2,h)$ be the corresponding abelian ideal of $\gob$ and let $\Psi \subset \Phi^+$ be the set of roots of $\goa$. If $S \subset \Psi$ is an orthogonal subset and $\grb \in \Phi$ is such that $s_{\grd-\grb} \leq \grs_{\wh S}$, then $\grb \in \Psi$ as well. 
\end{proposition}

\begin{proof}
Let $\wh \grL^\vee = \Hom_{\mathbb Z}(\mathbb Z \wh \Phi, \mathbb Z)$ be the dual lattice of $\wh \Phi$. Fix the basis $\{\gro_\gra^\vee \; | \;\gra \in \wh \grD \}$ defined by $\langle \gro_\gra^\vee,\beta\rangle =\delta_{\alpha,\beta}$ (for $\alpha, \beta\in \wh \grD$), and let $\wh \grL^\vee_+$ be the semigroup generated by this basis.

By definition we have $\Psi = \{\gra \in \Phi \; | \; \gra(h) = 2\}$. We can associate to $h$ an element $\grl_h \in \wh \grL^\vee_+$, by setting $\grl_h = \sum_{\gra \in \grD} \gra(h)\, \gro_\gra^\vee$. For $\grb \in \Phi$, notice that
$$
	\langle \grb, \grl_h \rangle = \sum_{\gra \in \grD} \gra(h) \langle \grb, \gro_\gra^\vee \rangle = \grb(h).
$$
Thus
$$
	\langle \grd, \grl_h \rangle = \langle \gra_0 + \theta, \grl_h \rangle = \langle \theta, \grl_h \rangle = 2,
$$
and we see that $\langle \grd - \grb, \grl_h \rangle = 0$ if and only if $\grb \in \Psi$.

If $\alpha\in \wh \Delta$, let $\alpha^\vee\in \wh \Lambda^\vee$ be the corresponding simple coroot, defined by $\langle \beta, \alpha^\vee \rangle =\frac{2(\alpha,\beta)}{( \alpha,\alpha)}$.
Define a partial order on $\wh \grL^\vee$ by setting $x\leq y$ iff $y-x$ is a sum of simple coroots, and notice that this is compatible with the Bruhat order on $\wh W$ in the following sense: if $v,w \in \wh W$ and $v\leq w$, then $v\lambda\geq w\lambda$ for all $\lambda\in\wh \grL^\vee_+$.

Suppose now that $s_{\grd - \grb} \leq \grs_{\wh S}$. Then by the previous discussion $\grs_{\wh S}(\grl_h) \leq s_{\grd-\grb }(\grl_h) \leq \grl_h$. On the other hand, as already notice above, the condition $S\subset \Psi$ implies that  $\grs_{\wh S}(\grl_h) = \grl_h$. Therefore $s_{\grd - \grb} (\grl_h) = \grl_h$, and we get $\grb \in \Psi$.
\end{proof} 

\section{The closure of a nilpotent orbit of height 2 and its resolution}

From now on we will restrict our attention to the case of the nilpotent orbits of height 2. For completeness we recall in Table \ref{tab:orbite-nilp} the classification of these orbits in terms of weighted Dynkin diagram, however the classification will not be needed in the following.

Notice that $\calN_2$ is irreducible unless $G$ is of type $B_r$ (in which case it has 2 irreducible components) or $D_r$ (in which case it has 2 irreducible components if $r$ is odd, and 3 irreducible components if $r$ is even).

\begin{table}\caption{Nilpotent orbits of height 2}\label{tab:orbite-nilp}
\begin{center}
\begin{tabular}{|c||c|c|c|c|c|c|}
\hline
& $G$ & $\begin{array} {c} \text{weighted Dynkin}\\ \text{diagram} \end{array}$ & $\begin{array} {c} \text{partition/}\\ \text{BC label} \end{array}$ & $G_0/L$ & $\rk(\goa)$
\\ \hline \hline
1.1 & $A_{2r+l}$ & $(\underbrace{0\ldots0}_{r-1} 1 \underbrace{0\ldots0}_{l} 1 \underbrace{0\ldots0}_{r-1})$ & $\phantom{\Big(} (2^r, 1^{l+1}) \phantom{\Big)}$ & $(A_{2r-1}, \alpha_r)$ & $r$
\\ \hline
1.2 & $A_{2r-1}$ & $(\underbrace{0\ldots0}_{r-1} 2 \underbrace{0\ldots0}_{r-1})$ & $\phantom{\Big(} (2^r) \phantom{\Big)}$ & $(A_{2r-1}, \alpha_r)$ & $r$
\\ \hline
2.1 & $B_{r+1}$ & $(20\ldots0)$ & $\phantom{\Big(} (3, 1^{2r}) \phantom{\Big)}$ & $(B_{r+1}, \alpha_1)$  & $2$
\\ \hline
2.2 & $B_{2r+l}$ & $(\underbrace{0\ldots0}_{2r-1}1\underbrace{0\ldots0}_{l})$ & $\phantom{\Big(} (2^{2r}, 1^{2l+1}) \phantom{\Big)}$ & $(D_{2r}, \alpha_{2r})$ & $r$
\\ \hline
3.1 & $C_{r+l+1}$ & $(\underbrace{0\ldots0}_{r-1}1 \underbrace{0\ldots0}_{l+1})$ & $\phantom{\Big(} (2^r, 1^{2l+2}) \phantom{\Big)}$ & $(C_r, \alpha_r)$ & $r$
\\ \hline
3.2 & $C_r$ & $(0\ldots02)$ & $\phantom{\Big(}(2^r) \phantom{\Big)}$ & $(C_r, \alpha_r)$ & $r$
\\ \hline
4.1 & $D_r$ & $(20\ldots0)$ & $\phantom{\Big(} (3, 1^{2r-1}) \phantom{\Big)}$ & $(D_r, \alpha_1)$ & $2$
\\ \hline
4.2 & $D_{2r+l+2}$ & $(\underbrace{0\ldots0}_{2r-1}1 \underbrace{0\ldots0}_{l+2})$ & $\phantom{\Big(} (2^{2r}, 1^{2l+4}) \phantom{\Big)}$ & $(D_{2r}, \alpha_{2r})$ & $r$
\\ \hline
4.3 & $D_{2r+1}$ & $(0\ldots011)$ & $\phantom{\Big(} (2^{2r}, 1^2) \phantom{\Big)}$ & $(D_{2r}, \alpha_{2r})$ & $r$
\\ \hline
4.4 & $D_{2r}$ & $(0\ldots002)$ & $\phantom{\Big(} (2^{2r})^I \phantom{\Big)}$ & $(D_{2r}, \alpha_{2r})$ & $r$
\\ \hline
4.5 & $D_{2r}$ & $(0\ldots020)$ & $\phantom{\Big(} (2^{2r})^{I\!I} \phantom{\Big)}$ & $(D_{2r}, \alpha_{2r})$ & $r$
\\ \hline
5.1 & $E_6$ & $(010000)$ & $A_1$ & $(A_1, \alpha_1)$ & $1$
\\ \hline
5.2 & $E_6$ & $(100001)$ & $2A_1$ & $(D_5, \alpha_1)$ & $2$
\\ \hline
6.1 & $E_7$ & $(1000000)$ & $A_1$ & $(A_1, \alpha_1)$ & $1$
\\ \hline
6.2 & $E_7$ & $(0000010)$ & $2A_1$ & $(D_6, \alpha_1)$ & $2$
\\ \hline
6.3 & $E_7$ & $(0000002)$ & $3A_1''$ & $(E_7, \alpha_7)$ &  $3$
\\ \hline
7.1 & $E_8$ & $(00000001)$ & $A_1$ & $(A_1, \alpha_1)$ & $1$
\\ \hline
7.2 & $E_8$ & $(10000000)$ & $2A_1$ & $(D_8, \alpha_1)$ & $2$
\\ \hline
8.1 & $F_4$ & $(1000)$ & $A_1$ & $(A_1, \alpha_1)$ &  $1$
\\ \hline
8.2 & $F_4$ & $(0001)$ & $\phantom{\Big(} \wt A_1 \phantom{\Big)}$ & $(B_4, \alpha_1)$ & $2$
\\ \hline
9.1 & $G_2$ & $(01)$ & $A_1$ & $(A_1, \alpha_1)$ &  $1$
\\ \hline
\end{tabular}
\end{center}
\end{table}

Let $e \in \mathcal N_2$, and set $X = \overline{Ge}$. Let $\{e,h,f\}$ be an $\mathfrak{sl}_2$-triple with nilpositive part $e$, up to conjugation we may assume (and we will assume) that $h$ is the dominant characteristic of $Ge$. We denote by $P$ the parabolic subgroup of $G$ with Lie algebra $\gop = \bigoplus_{i \geq 0} \gog(i,h)$ and by $L$ the Levi subgroup of $P$ with Lie algebra $\gol = \gog(0,h)$. Moreover we denote by $\goa = \gog(2,h)$ the associated abelian ideal of $\gob$, and by $\Psi = \Phi(2,h)$ the corresponding set of positive roots. Finally we will denote by $\Ort(\Psi)$ the set of orthogonal subsets of $\Psi$: since $\goa$ is an abelian ideal, every orthogonal subset of $\Psi$ is actually strongly orthogonal.

By \cite{hesselink} the variety $X$ is normal, and the action of $G$ induces an equivariant projective, birational and surjective morphism
$$
	\phi : G \times_P \goa \longrightarrow G \goa = X,
$$
which is a rational resolution. We set $\wt X = G \times_P \goa$. Since $P^u$ acts trivially on $\goa$, it follows that $\wt X$ is obtained from the $L$-variety $\goa$ via parabolic induction. Given $S \in \Ort(\Psi)$, we will denote by $\tilde e_S = [1,e_S]$ the point defined by $e_S \in \goa$ inside $\wt X$.

\begin{remark}	\label{oss:symmetric}
When we regard $\goa$ as $L$-variety, it is convenient to consider the symmetric subgroup $G_0 \subset G$ with Lie algebra $\gog_0 = \gog(-2) \oplus \gog(0) \oplus \gog(2)$. Notice that $G_0 = G$ if and only if $e$ is an even nilpotent element. Otherwise $\gog_0$ is a reductive subalgebra of $\gog$, whose set of simple roots is $\grD_{G_0} = \grD_L \cup \{w_L (\theta) \}$ (thus $w_L (\theta)$ is the unique minimal root of $\Psi$). 
Then $\goa$ is the nilradical of the parabolic subalgebra $\gop_0 \subset \gog_0$ defined by $\gop_0 = \gog(0) \oplus \gog(2)$. The corresponding parabolic subgroup $P_0 \subset G_0$ is the maximal parabolic subgroup of $G_0$ associated to the simple root $w_L(\theta) \in \grD_{G_0}$. Notice that by construction $P_0$ is a parabolic subgroup of $G_0$ with abelian unipotent radical, and Levi decomposition $P_0 = L P_0^u$.

By replacing $G$ with $G_0$, it follows that the study of $\goa$ as $L$-variety reduces to the case where $L$ is a symmetric subgroup of Hermitian type of $G$, and $\goa$ is the nilradical of the corresponding standard parabolic subalgebra $\gop$. In particular, a description of the $L$-orbits in $\goa$ follows from the works of Richardson, R\"ohrle and Steinberg \cite{RRS} and of M\"uller, Rubenthaler and Schiffman \cite{MRS}. 

The Hermitian symmetric space $G_0/L$ is actually of a particular type. Let indeed $w_{G_0} \in W_{G_0}$ be the longest element and let $(w_{G_0})^{P_0} \in W_{G_0}^{P_0}$ be the minimal length coset representative of the coset $w_{G_0} W_L$, then
$$
	w_{G_0}(w_L(\theta)) = w_L (w_{G_0})^{P_0} w_L (\theta) = w_L w_{G_0}(\theta) = - w_L(\theta).
$$ 
This means that $G_0/L$ is a Hermitian symmetric variety of \textit{nonexceptional type}, or of \textit{tube type} (see e.g. \cite[Proposition 1.5]{Pa4}, \cite[Proposition 3.1]{CM} and \cite[Section 4.4]{GMP}, and the references therein).

For any $G$-orbit in $\calN_2$, the Hermitian symmetric space $G_0/L$  is described in Table \ref{tab:orbite-nilp} by giving the type of the simple factor of $G_0$ acting on it, and its corresponding simple root.
\end{remark}

\subsection{Orbit decompositions of $X$, $\wt X$, and $\goa$}
In this subsection we will basically recall known facts, concerning the description of the $G$-orbits and the $B$-orbits in $X$ and in $\wt X$, and parallelly of the $P$-orbits and the $B$-orbits in $\goa$. 

As $P^u$ acts trivially on $\goa$, notice that the $P$-orbits on $\goa$ coincide with the $L$-orbits, and the $B$-orbits on $\goa$ coincide with the $B_L$-orbits.

\begin{proposition} 	\label{prop:G-orbite-vs-L-orbite}
Let $x \in \goa$, then $Gx \cap \goa = L x$.
\end{proposition}

\begin{proof}
Let $y \in \goa$ and $g \in G$ be such that $gx = y$. Denote $Q = P^-$, then we may assume that $g  \in Q^u$. Indeed $P Q^u$ and  $Q^u P$ are both open and dense in $G$, thus $G = (P Q^u)  (Q^u P) = P Q^u$: writing $g = phq$ with $p,q \in P$ and $h \in Q^u$, we may thus replace $x,y$ with $qx$ and $p^{-1}y$, which are still in $\goa$.

Thus we are reduced to show that $Q^u x\cap \goa=\{x\}$ for all $x \in \goa$. Let $x \in \goa$ and $g \in Q^u$, from the formula for the action of a root subgroup on a root vector it easily follows that
$$
	gx - x \in \gog(-2) \oplus \gog(-1) \oplus \gog(0) \oplus \gog(1).
$$
In particular we see that $gx \in \goa$ if and only if $gx = x$, and the claim is proved.
\end{proof}

\begin{corollary}
The resolution $\phi : \wt X \rightarrow X$ induces a bijection between the $G$-orbits on $\wt X$ and the $G$-orbits on $X$.
\end{corollary}

\begin{proof}
Since $P^u$ acts trivially on $\goa$, the $G$ orbits on $\wt X$ are in bijection with the $L$-orbits on $\goa$. On the other hand, thanks to Proposition \ref{prop:G-orbite-vs-L-orbite} the $G$-orbits on $X$ are also in bijection with the $L$-orbits on $\goa$.
\end{proof}

Thanks to Remark \ref{oss:symmetric}, a description of the $L$-orbit structure of $\goa$ follows from \cite{MRS} and \cite{RRS}. As shown in \cite{RRS} (see Propositions 2.13 and 2.15 therein), every $L$-orbit in $\goa$ is of shape $Le_S$ with $S \in \Ort(\Psi \cap \Phi_\ell)$, and given $S, S' \in \Ort(\Psi \cap \Phi_\ell)$ it holds
$$
	\overline{L e_S} \supset L e_{S'} \Longleftrightarrow \card(S) \geq \card(S').
$$
The \textit{rank} of $\goa$ is by definition the maximal cardinality of an orthogonal subset of $\Psi \cap \Phi_\ell$. In particular, if $r = \rk(\goa)$, then there are precisely $r+1$ orbits of $L$ in $\goa$, and they are linearly ordered by the inclusion of their closures. For any $G$-orbit in $\calN_2$, the rank of the corresponding abelian ideal is given in the last column of Table \ref{tab:orbite-nilp}.

The $B$-orbits on $\goa$ were investigated by Panyushev in \cite{Pa3}. In particular, there it is proved the following theorem (more generally, an analogous statement is proved in the case of any abelian ideal of $\gob$).

\begin{theorem}[\cite{Pa3}]
The $B$-orbits on $\goa$ are parametrized by the orthogonal subsets of $\Psi$, via the map $S \mapsto B e_S$.
\end{theorem}

As a corollary, we get the following description of the $B$-orbits on $\wt X$.

\begin{corollary}	\label{cor:orbite-in-tildeX}
The $B$-orbits on $\wt X$ are parametrized by $W^P \times \Ort(\Psi)$, via the map $(w,S) \mapsto B w \tilde e_S$.
\end{corollary}

\begin{proof}
Let $\calO \subset \wt X$ be a $B$-orbit and let $w \in W^P$ be the element defined by $\calO$ via the projection $\wt X \rightarrow G/P$. Then we can write $\calO = B [w, g]$ for some $g \in \goa$, and $w \in W^P$ is uniquely determined by this property. As $w \in W^P$, we have $B w = BwB_L$. On the other hand $P^u$ acts trivially on $\goa$, therefore we see that $\calO = B[w,B_L g] = BwB_L \tilde e_S = Bw\tilde e_S$ for some $S \in \Ort(\Psi)$.

If $S' \in \Ort(\Psi)$ and $Bw\tilde e_S = Bw\tilde e_{S'}$, then
$$\tilde e_{S'} \in (w^{-1}Bw) \tilde e_S \cap \goa =(B \cap w^{-1}.B) \, \tilde e_S \cap \goa.$$
Since $B_L \subset B \cap w^{-1}.B$ and since $B_L e_S = B e_S$, it follows that there exists a unique $S \in \Ort(\Psi)$ such that $\calO = Bw \tilde e_S$.
\end{proof}

By the previous corollary we get a similar parametrization for the $B$-orbits on $X$ as well.
Denote
$$
	\Ort(X) = \{w(S) \; | \; w \in W, \, S \in \Ort(\Psi)\}
$$
Notice that every element in $\Ort(X)$ is actually strongly orthogonal, since every element in $\Ort(\Psi)$ is so.

\begin{proposition}	\label{prop:B-orbits-in-X}
The $B$-orbits on $X$ are parametrized by $\Ort(X)$, via the map $S \mapsto Be_S$.
\end{proposition}

\begin{proof}
By Corollary \ref{cor:orbite-in-tildeX}, every $B$-orbit in $\wt X$ is of the shape $Bw \tilde e_S$ for some $w \in W$ and some $S \in \Ort(\Psi)$. Thus every $B$-orbit in $X$ is of the shape $Bw e_S$ for some $w \in W$ and some $S \in \Ort(\Psi)$. On the other hand, by Proposition \ref{prop:bruhat-involuzioni} and Corollary \ref{cor:iniettivita-involuzioni}, if $v,w \in W$ and $R,S \in \Ort(\Psi)$, then the orbits $Bwe_S$ and $Bve_R$ are equal if and only if $w(S) = v(R)$.
\end{proof}

\begin{corollary}	\label{cor:height2-is-orthogonal}
Every $B$-orbit in $\calN_2$ is strongly orthogonal.
\end{corollary}

We now introduce a numerical invariant that controls what are the orthogonal subsets of $\Psi$ which give rise to a same $L$-orbit. For $S \in \Ort(\Psi)$ define
$$
	\rk_G(S) = 2\card(S\cap \Phi_s) + \card(S\cap \Phi_\ell)
$$
The following proposition is essentially taken from \cite{RRS}.

\begin{proposition} 	\label{prop:L-orbite-in-a}
Let $S,R \in \Ort(\Psi)$, then $L e_S = L e_R$ if and only if $\rk_G(S) = \rk_G(R)$. Moreover, $e_S$ is in the open $L$-orbit of $\goa$ if and only if $S \in \Ort(\Psi)$ is maximal with respect to the inclusion.
\end{proposition}

\begin{proof}
By Remark \ref{oss:symmetric}, in order to study the $L$-orbits on $\goa$ we may assume that $\goa$ is the nilradical of $P$, which is therefore a parabolic subgroup of $G$ with abelian unipotent radical. By \cite[Proposition 2.8 and Remark]{RRS}, for all $p,q \in \mN$, the Weyl group $W_L$ acts transitively on all orthogonal subsets of $\Psi$ containing $p$ short roots and $q$ long roots. 

The claim follows immediately if all the involved roots are long. To conclude it is enough to show that, if $\Phi$ is not simply laced and $S$ contains precisely $p$ short roots and $q$ long roots, then there exists $S' \in \Ort(\Psi)$ containing $p-1$ short roots and $q+2$ long roots such that $e_{S'} \in L e_S$.

Suppose that $S$ contains a short root. Recall that the unipotent radical of a standard parabolic subgroup $Q \subset G$ is abelian if and only if $Q$ is maximal, corresponding to a simple root which occurs in the highest root $\theta$ with coefficient one. In particular, if $\Phi$ is not simply laced, then it is either of type $B$ or of type $C$, in which cases $Q$ is the maximal parabolic defined by $\gra_1$ in the first case and by $\gra_n$ in the second case (see \cite[Remark 2.3]{RRS}). Acting with $w_L$ we can assume that $S$ contains the highest short root $\theta_s$. Set $\gra = \theta - \theta_s$ and $\grb = 2\theta_s - \theta$, and notice that $\gra \in \Phi_L$ and $\grb \in \Psi$. Since we only have to deal with two cases, for simplicity we describe the roots through the usual $\gre$-notation.

Suppose that $\Phi$ is of type $B_n$. Then $\goa$ is the abelian nilradical associated to the maximal parabolic subgroup of $G$ defined by $\gra_1$, namely we have
$$
	\Psi = \{\gre_1, \gre_1 \pm \gre_2, \ldots, \gre_1 \pm \gre_n\}.
$$
In this case, $\theta_s = \gre_1$, $\theta = \gre_1+ \gre_2$, $\gra = \gre_2$ and $\grb = \gre_1 - \gre_2$. In particular, since there is no root in $\Psi$ orthogonal to $\theta_s$, we have $S = \{\theta_s\}$.

Suppose instead that $\Phi$ is of type $C_n$. Then $\goa$ is the abelian nilradical associated to the maximal parabolic subgroup of $G$ defined by $\gra_n$, namely we have
$$
	\Psi = \{\gre_i + \gre_j \; | \; 1 \leq i \leq j \leq n\}.	
$$
In this case, $\theta_s = \gre_1+\gre_2$, $\theta = 2\gre_1$, $\gra = \gre_1-\gre_2$ and $\grb = 2\gre_2$. In particular, we see that every root in $\Psi$ which is orthogonal to $\theta_s$ is necessarily orthogonal both with $\theta$ and with $\grb$.

Denote $S' = (S \setminus \{\theta_s\}) \cup \{\theta, \grb\}$. By the previous remarks $S' \in \Ort(\Psi)$, and by construction $S'$ contains $p-1$ short roots and $q+2$ long roots. On the other hand it is easy now to see that $e_{S'} \in U_{-\gra} U_\gra e_S \subset L e_S$, which concludes the proof of the first claim.

For the last claim, see e.g. \cite[Remark 1, pg. 333]{GMP}.
\end{proof}

\subsection{Canonical orthogonal subsets and the corresponding  characteristics}

A special orthogonal subsets of $\Psi$ can be constructed with the recursive cascade procedure, which goes back to Harish-Chandra. Following \cite{Pa4}, we will call it the \textit{canonical subset} of $\Psi$, and denote it by $S_\Psi$. We recall its construction, for details see e.g. \cite[Section 1]{Pa4}.

Let $\grg_1 \in \Psi$ be the unique maximal root (namely $\grg_1 = \theta$ is the highest root of $\Phi$) and set $S_\Psi^1 = \{\grg_1\}$. Assuming that $S_\Psi^i = \{\grg_1, \ldots, \grg_i\}$ has already been constructed, there exists a unique root $\grg_{i+1} \in \Psi$ which is orthogonal to all $\grg_j$ with $j \leq i$ and which is maximal with these properties, and
we set $S_\Psi^{i+1} = S_\Psi^i \cup \{\grg_{i+1}\}$.
When the procedure ends, the resulting set of roots is the canonical subset $S_\Psi$. In particular, it follows from the construction that $\grg_1 > \grg_2 > \ldots > \grg_r$.

Equivalently, the canonical subset $S_\Psi$ can be constructed from a similar construction starting from the unique minimal root $\grb_1 \in \Psi$ (that is, $\grb_1 = w_L(\theta)$) and set $S_{\Psi,1} = \{\grb_1\}$. Assuming that $S_{\Psi,i} = \{\grb_1, \ldots, \grb_i\}$ has already been constructed, there exists a unique root $\grb_{i+1} \in \Psi$ which is orthogonal to all $\grb_j$ with $j \leq i$ and which is minimal with these properties,
and we set $S_{\Psi, i+1} = S_{\Psi,i} \cup \{\grb_{i+1}\}$.
When the procedure ends, the resulting set of roots is again the canonical subset $S_\Psi$: namely for all $i \leq r$ it holds $\grb_i = \grg_{r+1-i}$. 

\begin{remark}	\label{oss:lower-orthogonal}
The fact that the two constructions produce the same orthogonal subset provides indeed a characterization of the Hermitian symmetric spaces of tube type, see \cite[Proposition 1.5]{Pa4} and its remarks.
\end{remark}

It is well known that all roots in $S_\Psi$ are long, and that $S_\Psi$ is an orthogonal subset of $\Psi$ of maximal cardinality. Thus by Proposition \ref{prop:L-orbite-in-a} the corresponding nilpotent element $e_{S_\Psi}$ is in the open $L$-orbit of $\goa$ (it actually belongs to the open $B$-orbit of $\goa$, see \cite[Lemma 2.5]{Pa3} and Remark \ref{oss:lower-orthogonal} above).

For $i \leq r$, we define
\begin{equation}	\label{eq:sl2-tripla canonica}
	e_i = \sum_{j=1}^i e_{\grg_j}, \qquad
	h_i = \sum_{j=1}^i h_{\grg_j}, \qquad
	f_i = \sum_{j=1}^i f_{\grg_j}.
\end{equation}

As a consequence of Proposition \ref{prop:L-orbite-in-a} we see that the elements $\{e_1, \ldots, e_r\}$ form a complete system of representatives for the $L$-orbits on $\goa$, hence for the $G$-orbits on $X$ by Proposition \ref{prop:G-orbite-vs-L-orbite}.

\begin{proposition}	\label{prop:caratteristiche-upper-canonical}
Let $i \leq r$, then $h_i$ is the dominant characteristic of $Ge_i$.
\end{proposition}

\begin{proof}
Since $\{e_i, h_i, f_i\}$ is an $\mathfrak{sl}_2$-triple and $h_i \in \got$, we only have to show that $h_i$ is dominant. This is easily shown by induction on $i$.

Let indeed $\gra \in \Phi^+$. Then $\grg_1= \theta$ is the highest root, thus $\gra(h_{\grg_1}) = \langle \gra, \grg_1^\vee \rangle \geq 0$. Suppose now that $\sum_{j=1}^{i-1} h_{\grg_j}$ is dominant, and let $\gra \in \Phi^+$. Then
$$
\gra(\sum_{j=1}^i h_{\grg_j}) = \sum_{j=1}^{i-1} \langle \gra, \grg_j^\vee \rangle + \langle \gra, \grg_i^\vee \rangle
$$
If $\langle \gra, \grg_i^\vee \rangle \geq 0$, then the claim follows by the inductive hypothesis. If instead $\langle \gra, \grg_i^\vee \rangle < 0$, then $\grg_j + \gra \in \Psi$ is greater than $\grg_j$, thus by the definition of $\grg_j$ the set $\{\grg_1, \ldots, \grg_{i-1}, \grg_i + \gra\}$ cannot be orthogonal. Equivalently,
$$
\sum_{j=1}^{i-1} \langle \grg_i + \gra, \grg_j^\vee \rangle = \sum_{j=1}^{i-1} \langle \gra, \grg_j^\vee \rangle \neq 0,	
$$
which by the inductive assumption means that
$\sum_{j=1}^{i-1} \langle \gra, \grg_j^\vee \rangle > 0$. On the other hand $\grg_i$ is long, thus $\langle \gra, \grg_i^\vee \rangle = -1$ and the claim follows.
\end{proof}

We now deduce some consequences of the previous constructions, relating the abelian ideals associated to the nilpotent elements $e_1, \ldots, e_r$. For $i \leq r$, let $\goa_i = \gog(2,h_i)$ be be the abelian ideal of $\gob$ defined by $h_i$ and let $\Psi_i = \Phi(2,h_i)$ be the corresponding set of positive roots.

\begin{proposition}	\label{prop:ideali-combinatorici-associati}
Let $i \leq r$, then
$$
	\Psi_i = 	\{\gra \in \Psi \; | \; \langle \gra, \grg_j^\vee \rangle = 0 \; \; \forall j > i \} = \{\gra \in \Psi \; | \; \gra \geq \grg_i \}.
$$
\end{proposition}

\begin{proof}
Notice first of all that, if $\gra \in \Psi_i$, then $\langle \gra, \grg_{i+1}^\vee \rangle = 0$.
Indeed, we have by definition $\gra(h_{i+1}) = \gra(h_i) + \langle \gra, \grg_{i+1}^\vee \rangle$. Since $e_{i+1}$ has height 2 and $h_{i+1}$ is the dominant characteristic of $Ge_{i+1}$, the assumption $\gra(h_i) = 2$ implies that $\langle \gra, \grg_{i+1}^\vee \rangle \leq 0$. Suppose that $\langle \gra, \grg_{i+1}^\vee \rangle < 0$: then $\gra + \grg_{i+1}$ is a root, thus $\gra + \grg_{i+1} \in \Psi_{i+1}$.  On the other hand $\grg_{i+1}$ is long, thus $\langle \gra, \grg_{i+1}^\vee \rangle = -1$ and we get 
$$
	(\gra + \grg_{i+1}) (h_{i+1}) = \gra(h_i) + \grg_{i+1}(h_{i+1}) + \langle \gra, \grg_{i+1}^\vee \rangle  = 3, 
$$
contradicting that $\height(e_{i+1}) = 2$.

Thanks to the previous discussion, we see that $\Psi_i \subset \Psi_{i+1}$, and that $\langle \gra, \grg_j^\vee \rangle = 0$ whenever $\gra \in \Psi_i$ and $j > i$. On the other hand $\grg_i (h_i) = 2$, thus $\grg_i \in \Psi_i$. Thus the claim follows by Remark \ref{oss:lower-orthogonal}, as $\grg_i$ is the unique minimal element in
\[
	\{\gra \in \Psi \; | \; \langle \gra, \grg_{i+1}^\vee \rangle = \ldots = \langle \gra, \grg_r^\vee \rangle = 0  \}.
	\qedhere
\] 
\end{proof}

\begin{corollary} \label{ideali uguali}
Let $e' \in \calN_2$ and let $\goa' \subset \gob$ be the abelian ideal defined by the dominant characteristic of $Ge'$. Then $Ge' \subset \overline{Ge}$ if and only if $\goa' \subset \goa$.
\end{corollary}

\begin{proof}
One implication is clear from the equalities $\overline{Ge} = G \goa$ and $\overline{Ge'} = G \goa'$. The other one follows from Proposition \ref{prop:ideali-combinatorici-associati}, as the elements $e_1, \ldots, e_r$ are a complete system of representatives for the $G$-orbits in $X$, and by Proposition \ref{prop:caratteristiche-upper-canonical} the corresponding dominant characteristics are $h_1, \ldots, h_r$.
\end{proof}

As we have seen in Proposition \ref{prop:caratteristiche-upper-canonical}, the canonical orthogonal subset $S_\Psi$ gives rise to the dominant characteristic of $Ge$. We now show that the same is true if we start from any maximal orthogonal subset of $\Psi$ (that is, an orthogonal subset of $\Psi$ which is maximal with respect to the inclusion).

\begin{proposition} \label{prop:caratteristica}
Let $S \subset \Psi$ be maximal orthogonal, then $h_S = h_r$ is the dominant characteristic of $Ge$.
\end{proposition}

\begin{proof}
To show the statement we will freely make use of some well known facts concerning the $\mathfrak{sl}_2$-triples in $\gog$ and the structure of the centralizer of a nilpotent element, for details see e.g. \cite[Section 3.4 and 3.7]{CMcG}.

By Proposition \ref{prop:L-orbite-in-a}, $e_S$ is in the open $L$-orbit $L e_r \subset \goa$ if and only if $\rk_G(S)$ is maximal, if and only if $S$ is a maximal orthogonal subset. Suppose that $e_S = ge_r$ with $g \in L$ and set $h'_S = gh_r$ and $f'_S = gf_r$, then $\{e_S, h'_S, f'_S\}$ is an $\mathfrak{sl}_2$-triple containing $e_S$ as the nilpositive element. Since all the eigenspaces $\gog(i,h_r)$ are $L$-stable and since $g \in L$, it follows that $h'_S$ induces the same grading as $h_r$. Since $\{e_S, h_S, f_S\}$ is also an $\mathfrak{sl}_2$-triple containing $e_S$ as the nilpositive element, by a theorem of Kostant there exists $v \in C_G(e_S)^u$ such that $vh'_S = h_S$ and $vf'_S = f_S$. Denote $C = C_G(e_S)$ and set
$$\goc(i) = \goc \cap \gog(i,h_S') = \goc \cap \gog(i,h_r).$$
Then $\goc = \bigoplus_{i \geq 0} \goc(i)$ and $\goc^u = \bigoplus_{i > 0} \goc(i)$, thus $h_S - h'_S \in \bigoplus_{i > 0} \gog(i,h_r)$. Since by construction both $h'_S$ and $h_S$ are in $\gog(0,h)$ it follows that $h'_S = h_S$.

Thus $h_S$ and $h_r$ induce the same grading on $\gog$, and it follows that $\gra(h_S) = \gra(h_r)$ for all $\gra \in \grD$.  Therefore $h_S$ is also a dominant characteristic for $Ge_r$, and the claim follows by the uniqueness of the dominant characteristic.
\end{proof}

\begin{corollary}	\label{cor:max-orthogonal}
Let $S \subset \Psi$ be maximal orthogonal, then $S$ is maximal orthogonal in $\Phi^+ \setminus \Phi_L$ as well.
\end{corollary}

\begin{proof}
By Proposition \ref{prop:caratteristica} we have $h_S = h_r$, thus   $\gra(h_S) = 1$ for all $\gra \in \Phi(1)$.
\end{proof}

We now associate to any of the dominant characteristics $h_i$ a Levi subgroup of $G$ which will be fundamental in order to study the fibers of the resolution $\phi : \wt X \rightarrow X$.

Given $i \leq r$, let $P_i \subset G$ be the corresponding parabolic subgroup defined by $h_i$ and $L_i \subset P_i$ the corresponding Levi factor, with set of simple roots
$$
	\grD_{L_i} = \{\gra \in \grD \; | \; \gra(h_i) = 0\}.
$$
Let also $G_i \subset G$ be the symmetric subgroup of $G$ containing $L_i$ constructed as in Remark \ref{oss:symmetric}. By Proposition \ref{prop:ideali-combinatorici-associati}, we see that the set of simple roots of $G_i$ is 
$$
	\grD_{G_i} = \grD_{L_i} \cup \{\grg_i\}
$$
Let $\grD_{G_i}(\grg_i) \subset \grD_{G_i}$ be the subset corresponding to the connected component of the Dynkin diagram of $G_i$ containing $\grg_i$, and let
$$
	\grD_{L_i}^* = \grD_{L_i} \setminus \grD_{G_i}(\grg_i)
$$

Notice that $\grg_1$, which is the highest root of $\Phi$, also corresponds to the highest root of $\grD_{G_i}(\grg_i)$. Since $\grg_i$ occurs in $\grg_1$ with coefficient 1 as a simple root of $\grD_{G_i}$, we see that
$$
	\supp(\grg_1 - \grg_i) = \grD_{L_i} \cap \grD_{G_i}(\grg_i).
$$
As $\grg_i > \grg_r$, we obtain the inclusion
$$	\grD_{L_i} \setminus \grD_{L_i}^* \subset \grD_{L_r} \setminus \grD_{L_r}^* \subset \grD_{L_r} = \grD_L. \\
$$

We denote by $L_i^* \subset G$ the Levi subgroup associated to the subset $\grD_{L_i}^* \subset \grD$. Since $\Psi_i$ is contained in the set of positive roots of $G_i$ corresponding to the connected component $\grD_{G_i}(\grg_i) \subset \grD_{G_i}$, we see that for all $\gra \in \grD_{L_i}^*$ the root vectors $e_\gra$ and $f_\gra$ are in the Lie algebra of the centralizer $C_{L_i}(\goa_i)$.

Summarizing the previous discussions, we get the following.

\begin{proposition} \label{prop:centralizzatore-ridotto}
Denote $C_i = C_{L_i}(\goa_i)^\circ$. Then $C_i$ is a normal reductive subgroup of $L_i$, and $B \cap C_i$ is a Borel subgroup therein. Moreover,
$$L_i/L_i \cap P  \; \simeq \; L_i^*/L_i^* \cap P \; \simeq \; C_i/C_i \cap P.$$
\end{proposition}

\subsection{The fibers of the resolution.}

In this subsection we will study the fibers of the resolution $\phi : \wt X \rightarrow X$, and will show that they are flag varieties under the action  of a suitable subgroup of the centralizer. As $X$ is normal,  the fibers of $\phi$ are connected  by Zariski's main theorem, and complete because $\phi$ is projective.

Regard $\wt X$ as a closed subvariety of the trivial bundle $G \times_P \calN \simeq G/P \times \calN$, where the latter isomorphism is given by $[g,z] \mapsto (gP,gz)$. Thus we have a commutative diagram
$$
	\xymatrix{
	\wt X \ar@{->}[r]\ar@{->}[d]_{\phi} & G/P\times \calN \ar@{->}[d] \\
	X \ar@{->}[r] & \calN \\
	}
$$
In particular, for $x \in X$, the projection $\pi : \wt X \rightarrow G/P$ restricts to a closed embedding $\pi_x : \phi^{-1}(x) \rightarrow G/P$.

Notice that $\phi^{-1}(x)$ is a flag variety under the action of $C_G(x)$. If indeed $x \in \goa$, then $Gx \cap \goa = Lx$ thanks to Proposition \ref{prop:L-orbite-in-a}, thus
$$
	\phi^{-1}(x) = \{[g,y] \; | \; gy = x\} = \{[g, x] \;|\; g \in C_G(x)\} \simeq C_G(x) / C_P(x).
$$
Therefore, if $x \in X$ and $g x \in \goa$, we get
$$
	\phi^{-1}(x) \; \simeq \; C_G(x)/C_{g.P}(x) 
$$

We now study the fibers $\phi^{-1}(e_i)$ for $i \leq r$. By standard results on the centralizers of nilpotent elements and the associated characteristics (see e.g. \cite[Section 3.7]{CMcG}), we have $C_G(e_i) = C_{L_i}(e_i) \, C_G(e_i)^u$ and $C_G(e_i)^u \subset P_i^u$.  Setting $C_i = C_{L_i}(\goa_i)^\circ$, by Proposition \ref{prop:centralizzatore-ridotto} 
we get then
\begin{equation}	\label{eq:fibra}
	\phi^{-1}(e_i) \; \simeq \; C_G(e_i) / C_P(e_i) \; \simeq \; C_{L_i}(e_i) / C_{L_i}(e_i) \cap P \; \simeq \; C_i / C_i \cap P.
\end{equation}
Thus $\phi^{-1}(e_i)$ is a flag variety under the action of $C_i$, and again by Proposition \ref{prop:centralizzatore-ridotto}  we see that
$$
	B \, \backslash \, \phi^{-1}(Be_i) \; \simeq B \cap C_i \, \backslash \, \phi^{-1}(e_i) \; \simeq \; B \cap C_i \, \backslash \, C_i \, / \, C_i \cap P 
$$
is identified with the set of the Schubert cells in a partial flag variety for the reductive group $C_i$.

More generally, we have the following. 

\begin{proposition}	\label{prop:fibre-bandiere}
Let $i \leq r$, let $w \in W^{P_i}$ and $g \in L_i$. Then $\phi^{-1}(wge_i)$ is a flag variety under $w.C_i$, and $B \cap w. C_i  = w.(B \cap C_i)$ is a Borel subgroup of $w.C_i$.
\end{proposition}

\begin{proof}
Since $C_i$ is normal in $L_i$, equation \eqref{eq:fibra} together with Proposition \ref{prop:centralizzatore-ridotto} imply that $\phi^{-1}(wge_i)$ is a flag variety under $(w g).C_i = w.C_i$.

Since $w \in W^{P_i}$, we have $w.(B \cap C_i) \subset w.(B \cap L_i) \cap w.C_i \subset B \cap w.C_i$. Thus by Proposition \ref{prop:centralizzatore-ridotto} we see that $B \cap w.C_i$ contains a Borel subgroup of $w.C_i$. On the other hand $B \cap w.C_i$ is a solvable group, hence $B \cap w.C_i = w.(B \cap C_i)$ is a Borel subgroup of $w.C_i$.
\end{proof}

Let $R \in \Ort(X)$. Then by Corollary \ref{cor:orbite-in-tildeX} and Proposition \ref{prop:B-orbits-in-X} we have
$$
	\phi^{-1}(B e_R) = \bigsqcup_{w(S) = R} B w \tilde e_S,
$$
where the union runs over all the pairs $(w,S) \in W^P \times \Ort(\Psi)$ such that $w(S) = R$. By abuse of notation, we also write
$$
	\phi^{-1}(R) = \{(w,S) \in W^P \times \Ort(\Psi) \; | \;  w(S) = R\}.
$$
We regard $\phi^{-1}(R)$ as a partially ordered set, with the order induced by inclusion of orbit closures in $\widetilde X$. 

\begin{corollary}	\label{cor:ammissibile}
Let $R \in \Ort(X)$, then the following hold.
\begin{itemize}
	\item[i)] $\phi^{-1}(B e_R)$ contains a unique closed $B$-orbit.
	\item[ii)] There exists a unique pair $(w,S) \in \phi^{-1}(R)$ such that $w$ has minimal length, in which case $Bw \tilde e_S$ is the unique closed $B$-orbit in $\phi^{-1}(B e_R)$.
	\item[iii)]	There exists a unique element of minimal length $w \in W^P$ such that $w^{-1}(R) \subset \Psi$, in which case
	$Bw \tilde e_{w^{-1}(R)}$ is the unique closed $B$-orbit inside $\phi^{-1}(B e_R)$.
\end{itemize}
\end{corollary}

\begin{proof}
i) Let $i \leq r$ be such that $e_R \in Ge_i$, and let $w \in W^{P_i}$ and $g \in L_i$ be such that 
$e_R = w g e_i$. Then by Proposition \ref{prop:fibre-bandiere} the fiber $\phi^{-1}(e_R)$ is homogeneous under $w.C_i$, and $B \cap w.C_i$ is a Borel subgroup in $w.C_i$. Since
$$
	B \backslash \phi^{-1}(B e_R) \simeq (B \cap w.C_i) \backslash \phi^{-1}(e_R)
$$
and $\phi^{-1}(e_R)$ is a flag variety for $w.C_i$, it follows in particular that $\phi^{-1}(B e_R)$ contains a unique closed $B$-orbit. Thus we have proved the first statement.

ii) Notice that every element $(w,S) \in \phi^{-1}(R)$ is uniquely determined by its first component $w$, via the equality $S = w^{-1}(R)$. On the other hand, projecting on $G/P$ we see that if $Bw \tilde e_S \subset \overline{ Bw' \tilde e_{S'}}$ then $w \leq w'$ as well. Thus the claim follows from i).

iii) It follows from ii), by noticing that
\[
	\phi^{-1}(R) = \{(w,w^{-1}(R)) \; | \; w \in W^P \text{ and } w^{-1}(R) \subset \Psi.\}
	\qedhere
\]
\end{proof}

\begin{definition}
Let $w \in W^P$ and $S \in \Ort(\Psi)$. We say that the pair $(w,S)$ is \textit{admissible} if $Bw \tilde e_S$ is closed inside $\phi^{-1}(Bwe_S)$.
\end{definition}

As an immediate consequence of Corollary \ref{cor:ammissibile}, we see that the admissible pairs are in bijection with the $B$-orbits in $X$. Notice that a pair $(w,S)$ is always admissible if $S$ is maximal: indeed in this case $we_S \in Ge_r$ is in the open $G$-orbit of $X$, and since $\phi$ is birational it follows that $\phi^{-1}(Bwe_S) = Bw \tilde e_S$ is a single $B$-orbit.

In the description of the fiber $\phi^{-1}(R)$ we can be even more explicit.

\begin{corollary}	\label{cor:bruhat-fibre}
Let $R \in \Ort(X)$ and suppose that $e_R \in G e_i$. Let $(w,S)\in W^{P_i} \times \Ort(\Psi_i)$ be the unique pair such that $w(S) = R$, then the map 
$$
	W_{L_i^*}^{L_i^* \cap P} \longrightarrow \phi^{-1}(R) \qquad \qquad u \longmapsto \Big( (wu)^P, (wu)_P(S) \Big)
$$
is an order isomorphism.
\end{corollary}

\begin{proof}
By Proposition \ref{prop:fibre-bandiere} the fiber $\phi^{-1}(e_R)$ is a flag variety under $w.C_i$, and $B \cap w.C_i = w.(B \cap C_i)$ is a Borel subgroup of $w.C_i $. Therefore we get isomorphisms of partially ordered sets
$$
	\phi^{-1}(R) \; \simeq \; B \, \backslash \, \phi^{-1}(Be_R)  \; \simeq \; w.(B \cap C_i) \, \backslash \, w.C_i \, / \, w.(C_i \cap P) \; \simeq \; W_{C_i}^{C_i \cap P} 	\; \simeq \; W_{L_i^*}^{L_i^* \cap P} 
$$

The claim follows by noticing that, if we compose the previous isomorphisms with the natural injection $\phi^{-1}(R) \rightarrow W^P$ (that is, the projection on the first factor), then the induced map $W_{L_i^*}^{L_i^* \cap P} \rightarrow W^P$ is nothing but $u \mapsto (wu)^P$.
\end{proof}

\section{The Bruhat order on $\wt X$}

We keep the notation of the previous section. In particular $e \in \calN_2$ is a fixed nilpotent element contained in an $\mathrm{sl}_2$-triple $\{e,h,f\}$, where $h$ is the dominant characteristic for $Ge$, and $X = \overline{Ge}$. Moreover $\goa = \gog(2)$ is the associated abelian ideal of $\gob$ whose set of roots is denoted by $\Psi$, and $P$ is the standard parabolic subgroup with Lie algebra $\gop = \bigoplus_{i \geq 0} \gog(i)$ and Levi subgroup $L$ with Lie algebra $\gol = \gog(0)$. If moreover $r = \rk(\goa)$, then we can assume that $\{e,h,f\} = \{e_r, h_r, y_r\}$ is the $\mathfrak{sl}_2$-triple defined in \eqref{eq:sl2-tripla canonica}.

We now characterize the partial order on the $B$-orbits in the resolution $\wt X = G \times^P \goa$ in terms of the corresponding pairs in $W^P \times \Ort(\Psi)$. A first connection follows immediately by projecting orbits in $\wt X$ respectively to $G/P$ and to $X$.

\begin{proposition} \label{prop:implicazione-facile-sopra}
Let $v,w \in W^P$ and $R,S \in \Ort(\Psi)$. If $Bv\tilde e_R \subset \ol{Bw \tilde e_S}$, then $v \leq w$ and $\grs_{v(\wh R)} \leq \grs_{w(\wh S)}$.
\end{proposition}

\begin{proof}
The inequality $v \leq w$ follows applying the projection $\wt X \rightarrow G/P$, whereas $\grs_{v(\wh R)} \leq \grs_{w(\wh S)}$ follows from Proposition \ref{prop:bruhat-involuzioni} applying the projection $\wt X \rightarrow X$.
\end{proof}

As shown in \cite{GMMP}, the dimension of a $B$-orbit in $\goa$ is read off from the length of the corresponding involution. As a consequence, we also get a formula for the dimension of the $B$-orbit in $\wt X$.

\begin{definition}
Let $w \in W^P$ and $S \in \Ort(\Psi)$, the \textit{length} of $(w,S)$ is
$$
	L(w,S) = \ell(w) + L(\grs_{\wh S}).
$$
\end{definition}
Sometimes it will also be convenient to denote the involution $\grs_{w(\wh S)}$ by $\grs(w,S)$.

As $\dim (Bw \tilde e_S) = \ell(w) + \dim(B e_S)$, the dimension formula \cite[Corollary 5.4]{GMMP} implies that
\begin{equation}\label{dimorbitasopra}
	\dim (Bw \tilde e_S) =  \ell(w) + L(\grs_{\wh S}) = L(w,S).
\end{equation}

In the case of an admissible pair, we will see that $L(w,S) = L(\grs_{w(\wh S)})$ is the dimension of $Bwe_S \subset X$.

\begin{definition}
Let $w \in W^P$ and $S \in \mathrm{Ort}(\Psi)$. We say that $\gra \in \grD$ is a \textit{descent} for $(w,S)$ if either it is a descent for $\grs_{w(\wh S)}$, or if $s_\gra w < w$. In the latter case we say that $\gra$ is an \textit{external descent}.  If instead $\gra \in \grD$ is a descent for $(w,S)$ and $s_\gra w > w$ then we say that $\gra$ is an \textit{internal descent}.
\end{definition}

Given $\gra\in\Delta$, denote by $P_\gra$ the minimal parabolic subgroup of $G$ associated to $\gra$.

\begin{lemma}	\label{lemma:discese}
Let $(w,S) \in W^P \times \mathrm{Ort}(\Psi)$ and let $\gra \in \grD$ be a descent for $(w,S)$, then $\overline{Bw \tilde e_S}$ is $P_\gra$-stable. If moreover $\gra$ is internal and complex (resp. real), then $w^{-1}(\gra) \in \grD_L$ and it is a complex (resp. real) descent for $\grs_{\wh S}$.
\end{lemma}

\begin{proof}
If $\gra$ is an external descent, the claim follows from \cite[Lemma 6]{brion}: indeed in this case $P_\gra w \tilde e_S = B s_\gra w \tilde e_S \cup B w \tilde e_S$, and $B w \tilde e_S$ is open inside $P_\gra w \tilde e_S$ because $s_\gra w < w$.

Suppose now that $\gra$ is an internal descent. Then $s_\gra w > w$. Set  $\grb = w^{-1}(\gra)$ and note that, by assumption, $\grb \in \Phi^+$.
Suppose first that $\grb \in \Phi^+ \setminus \Phi_L$. Then for all $\grg \in \Psi$ we have $\grb + \grg \not \in \Phi$, hence $\langle \grb, \grg^\vee \rangle \geq 0$. Therefore
$$
	\grs_{w(\wh S)}(\gra) = \gra - \sum_{\grg \in S} \langle \grb, \grg^\vee \rangle \, w(\grg - \grd)
$$
is a positive root, a contradiction.

Therefore it must be $\grb \in \Phi_L$. As $w\in W^P$, notice that $\grb \in \grD_L$: if indeed $\grb=\grb_1+\grb_2$ with $\grb_1,\grb_2 \in \Phi_L^+$, then $\gra=w(\grb_1)+w(\grb_2)$ would be a sum of positive roots. We claim that $\grb$ is a descent for $\grs_{\wh S}$ as well. Otherwise it must be $\grs_{\wh S}(\grb) \in \Phi^+(w)$, hence from  the equality
$$\grs_{\wh S}(\grb) = \grb - \sum_{\grg \in S} \langle \grb, \grg^\vee \rangle \, (\grg - \grd)
$$
we get $\sum_{\grg \in S} \langle \grb, \grg^\vee \rangle = 0$. Thus there are $\grg_1, \ldots, \grg_{2n} \in S$ such that 
$$\grs_{\wh S}(\grb) = \grb + \sum_{i=1}^n (\grg_i - \grg_{i+1})$$
yielding $(\grs_{\wh S}(\grb))(h) = 0$. This shows that $\grs_{\wh S}(\grb) \in \Phi^+_L$, thus
$$\grs_{w(\wh S)}(\gra) = w(\grs_{\wh S}(\grb) ) \in w(\Phi^+_L) \subset \Phi^+,$$
a contradiction.

Thus we have shown that $\grb$ is a descent for $\grs_{\wh S}$. Hence by \cite[Proposition 6.1]{GMMP} we see that $\overline{Be_S} = \overline{B_L e_S}$ is $P_\grb$-stable. Thus
$$B s_\gra B w \tilde e_S = B s_\gra U_\gra w \tilde e_S \subset B w P_\grb \tilde e_S \subset \overline {B w B_L \tilde e_S} = \overline {B w \tilde e_S},$$
and it follows that $P_\gra w \tilde e_S \subset \overline{Bw \tilde e_S}$.

It only remains to prove that, if $\gra$ is an internal descent for $\grs_{w(\wh S)}$, then it is complex (resp. real) if and only if $\grb$ it is so as a descent for $\grs_{\wh S}$, but this is immediate. 
\end{proof}

In the particular case of an internal descent, combining together Lemma \ref{lemma:discese} and \cite[Proposition 6.1]{GMMP} we get the following.

\begin{proposition}	\label{prop:discese interne}
Let $(w,S) \in W^P \times \mathrm{Ort}(\Psi)$, let $\gra \in \grD$ be an internal descent and set $\grb = w^{-1}(\gra) \in \grD_L$. Then the following hold:
\begin{itemize}
	\item[i)]	If $\gra$ is complex, then $w \tilde e_{s_\grb(S)} \in P_\gra w \tilde e_S$ and $s_\gra \circ \grs_{w(\wh S)} = \grs_{w s_\grb (\wh S)}$.
	\item[ii)] If $\gra$ is real, then there are unique elements $\grg_1,\grg_2 \in S$ such that $\grb = \tfrac{1}{2}(\grg_1 - \grg_2)$. Moreover, the set of roots
	\begin{equation}\label{sb}S_\grb := (S \setminus \{\grg_1, \grg_2\}) \cup \{\grb+\grg_2\}\end{equation}
	is an orthogonal subset of $\Psi$ such that $w \tilde e_{S_\grb} \in P_\gra w \tilde e_S$ and $s_\gra \circ \grs_{w(\wh S)} = \grs_{w(\wh S_\grb)}$.
\end{itemize}
\end{proposition}

We now consider the external descents of $(w,S)$.

\begin{proposition}	\label{prop:discese esterne}
Let $(w,S) \in W^P \times \mathrm{Ort}(\Psi)$, and let $\gra \in \grD$ be such that $s_\gra w < w$.
\begin{itemize}
	\item[i)] Either $\gra$ is a complex descent for $\grs_{w(\wh S)}$, or $w(S) \cup \{\gra\}$ is orthogonal.
	\item[ii)] If $(w,S)$ is admissible, then $\gra$ is a descent for $\grs_{w(\wh S)}$.
\end{itemize} 
\end{proposition}

\begin{proof}
i) Denote $\grb = - w^{-1}(\gra)$, then $\grb \in \Phi^+ \setminus \Phi_L$.  Since $\height(e) = 2$ and $\grb(h_r) > 0$, we have $\langle \grb, \grg^\vee \rangle \geq 0$ for all $\grg \in \Psi$: if indeed $\langle \grb, \grg^\vee \rangle < 0$ for some $\grg \in \Psi$, then $\grb+\grg$ would be a root and $(\grb + \grg)(h_r)>2$.
Write
$$
	\grs_{w(\wh S)}(\gra) = \gra + \sum_{\grg \in S} \langle \grb, \grg^\vee \rangle \, w(\grg - \grd).
$$
Since $\gra \in \grD$ and $w(\grg-\grd) \in \wh \Phi^-$ for all $\grg \in S$, it follows that $\gra$ is a descent for $\grs_{w(\wh S)}$ if and only if $\langle \grb, \grg^\vee \rangle \neq 0$ for some $\grg \in S$. In particular, if $\gra$ is not a descent for $\grs_{w(\wh S)}$, then $w(S) \cup \{\gra\}$ is an orthogonal set of roots. On the other hand, if $\gra$ is a descent for $\grs_{w(\wh S)}$, then
$$
	s_\gra \grs_{w(\wh S)} (\gra) = - \gra + \sum_{\grg \in S} \langle \grb, \grg^\vee \rangle \, s_\gra w(\grg - \grd)
$$
is still in $\wh \Phi^-$: thus in this case $\gra$ is necessarily a complex descent.

ii) By Corollary \ref{cor:max-orthogonal}, every maximal orthogonal subset of $\Psi$ is also maximal in $\Phi^+ \setminus \Phi_L$. It follows that, if $S$ is maximal, then $\gra$ is necessarily a descent for $\grs_{w(\wh S)}$. Suppose now that $S$ is not maximal, and let $i \leq r$ be such that $e_S \in G e_i$. Let $(v,R) \in W^{P_i} \times \Ort(\Psi_i)$ be the unique admissible pair (with respect to the orbit $Ge_i$) such that $w(S) = v(R)$. Then by Corollary \ref{cor:bruhat-fibre} we see that $w =v^P$: in particular, every left descent of $w$ is a left descent of $v$ as well. Therefore, from the case of a maximal orthogonal subset considered above, we deduce that $\gra$ is a descent for $\grs_{v(\wh R)} = \grs_{w(\wh S)}$.
\end{proof}

\begin{theorem}	\label{teo:dimension}
Let $(w,S) \in W^P \times \mathrm{Ort}(\Psi)$ be admissible, then
$$\dim (Bw \tilde e_S) = \dim (Bwe_S) = L(\grs_{w(\wh S)}).$$ 
\end{theorem}

\begin{proof}
We show the claim by induction on $\ell(w)$. The case $\ell(w) = 0$ follows from \cite[Theorem 5.3]{GMMP}. Assume that $\ell(w) > 0$ and let $\gra \in \grD$ be such that $s_\gra w < w$. Then $s_\gra w \in W^P$, and by Lemma \ref{lemma:discese} the orbit $B w \tilde e_S$ is dense inside $P_\gra w \tilde e_S$. Notice that $\dim(Bs_\gra w \tilde e_S)  = \dim(B w \tilde e_S) -1$ by \eqref{dimorbitasopra}.
Moreover by Proposition \ref{prop:discese esterne} we see that $\gra$ is a complex descent for $\grs_{w(\wh S)}$, thus $L(\grs_{s_\gra w(\wh S)}) = L(\grs_{w(\wh S)}) - 1$. 

By Corollary \ref{cor:ammissibile} we see that $w$ has minimal length among the elements in $W^P$ whose inverse moves $w(S)$ inside $\Psi$. This easily implies that $s_\gra w$ also has minimal length among the elements in $W^P$ whose inverse moves $s_\gra w(S)$ inside $\Psi$, hence by Corollary \ref{cor:ammissibile} again we see that the pair $(s_\gra w, S)$ is also admissible. Therefore by the inductive hypothesis we get
$$
	\dim(Bw \tilde e_S)  = \dim(Bs_\gra w \tilde e_S) +1 = L(\grs_{s_\gra w(\wh S)}) +1 =  L(\grs_{w(\wh S)}).
$$

It only remains to show that $\dim(Bw e_S)  = L(\grs_{w(\wh S)})$. If $S$ is a maximal orthogonal subset, this follows from the fact that $\phi$ is birational. If this is not the case, let $i \leq r$ be such that $e_S \in G e_i$ and let $(v,R) \in W^{P_i} \times \Ort(\Psi_i)$ be the unique admissible pair such that $v(R) = w(S)$. Then we can replace $\phi$ with the resolution $G \times_{P_i} \goa_i \rightarrow \overline{Ge_i}$, and since it is birational we obtain
$$
	\dim(Bwe_S) = \dim(Bv \tilde e_R) = L(\grs_{v(\wh R)}).
$$
On the other hand $v(R) = w(S)$, and the claim follows.
\end{proof}

The dimension formula in Theorem \ref{teo:dimension} only depends on the orthogonal subset $w(S)$, thus we can also rephrase it without referring to the admissible pair $(w,S)$. 

\begin{corollary}	\label{cor:dim-formula}
Let $S \subset \Phi$ be strongly orthogonal and suppose that $\height(e_S) = 2$, then $\dim (B e_S) = L(\grs_{\wh S})$.
\end{corollary}

Let $(w,S) \in W^P \times \mathrm{Ort}(\Psi)$ and let $\gra$ be a descent for $(w,S)$, and set $\grb = w^{-1}(\gra)$. Then we define
$$
	\mathcal F_\gra(w,S) = \left\{
	\begin{array}{ll}
	\big(s_\gra w, S \big) & \text{ if $\gra$ is external,}\\
		\big(w, S_\grb \big) & \text{ if $\gra$ is internal,}
	\end{array}
\right.
	$$
where, if $\gra$ is internal and complex, we have set $S_\grb = s_\grb(S)$. If instead $\gra$ is internal and real, then $S_\grb \subset \Psi$ is  the subset defined in \eqref{sb}.

\begin{remark}	\label{oss:figli}
Notice that
$$
	L(\calF_\gra (w, S)) = L(w,S) -1:
$$
indeed $\calF_\gra (w, S)$ corresponds to a $B$-orbit in $P_\gra w \tilde e_S \setminus Bw \tilde e_S $, and every such an orbit has codimension one in $P_\gra w \tilde e_S$. If moreover $\gra$ is a descent for $\grs_{w(\wh S)}$, then by Propositions \ref{prop:discese interne} and \ref{prop:discese esterne} we have
$$
	\grs(\calF_\gra (w, S)) = s_\gra \circ \grs_{w(\wh S)}.
$$
In particular, if $(w,S)$ is admissible, then $\calF_\gra (w, S)$ is admissible as well.
\end{remark}

\begin{theorem}	\label{teo:bruhat1}
Let $(v,R), (w,S) \in W^P \times \Ort(\Psi)$. Then $B v \tilde e_R \subset \overline {B w \tilde e_S}$ if and only if $v \leq w$ and $\grs_{v(\wh R)} \leq \grs_{w(\wh S)}$.
\end{theorem}

\begin{proof}
One implication has already been proved in Proposition \ref{prop:implicazione-facile-sopra}. Thus we only have to show that, if $v \leq w$ and $\grs_{v(\wh R)} \leq \grs_{w(\wh S)}$, then $Bv \tilde e_R \subset \overline{Bw \tilde e_S}$ as well.

We proceed by induction on $\ell(w)$. If $\ell(w) = 0$, then $\ell(v) = 0$ as well and the claim follows from \cite[Theorem 6.3]{GMMP}. 

Suppose that $\ell(w) > 0$. Let $\gra \in \grD$ be a descent for $w$ and consider the pair $(s_\gra w, S)$: then by Lemma \ref{lemma:discese} we get
$$B s_\gra w \tilde e_S \subset P_\gra w \tilde e_S \subset \ol{B w \tilde e_S}.$$
On the other hand, by Proposition \ref{prop:discese esterne} we see that either $\gra$ is a complex descent for $\grs_{w(\wh S)}$, or $s_\gra$ and $\grs_{w(\wh S)}$ commute. We distinguish three possible cases, according whether  $\gra$ is a descent for 
$(v,R)$ or not, and its type as a descent.

\textit{Case 1}. Suppose that $s_\gra v < v$, and consider the pair $(s_\gra v,R) = \calF_\gra(v,R)$. Then by Lemma \ref{lemma:discese} we have
$$B v \tilde e_R \subset P_\gra s_\gra v \tilde e_R \subset \ol{B v \tilde e_R}.$$
Notice that $s_\gra v \leq s_\gra w$ and $\grs_{s_\gra v(\wh R)} \leq \grs_{s_\gra w(\wh S)}$. Thus $B s_\gra v \tilde e_R\subset \overline{B s_\gra w \tilde e_S}$ by the inductive assumption, and we get
$$
	B v \tilde e_R \subset P_\gra s_\gra v \tilde e_R \subset \overline{P_\gra s_\gra w \tilde e_S} \subset \overline{Bw \tilde e_S}.
$$

\textit{Case 2}. Suppose that $\gra$ is an internal descent for $(v,R)$. Then $s_\gra v > v$ and $\grb=v^{-1}(\gra) \in \grD_L$. Consider the pair $(v,R_\grb) =  \calF_\gra(v,R)$; then
$$B v \tilde e_R \subset P_\gra v \tilde e_{R_\grb} \subset \ol{B v \tilde e_R}.$$
Notice that $v \leq s_\gra w$ and $\grs_{v(\wh R_\grb)} \leq \grs_{s_\gra w(\wh S)}$. Thus the inductive assumption yields
$$
	B v \tilde e_R \subset P_\gra v \tilde e_{R_\grb} \subset \overline{P_\gra s_\gra w \tilde e_S} \subset \overline{Bw \tilde e_S}.
$$

\textit{Case 3}. Suppose finally that $\gra$ is not a descent for $(v,R)$. Then $s_\gra v > v$, and thanks to Lemmas \ref{lemma:par1} and  \ref{lemma:par2} we have $v \leq s_\gra w$ and $\grs_{v(\wh R)} \leq \grs_{s_\gra w (\wh S)}$. Therefore by the inductive hypothesis we get
\[
	Bv \tilde e_R \subset \overline{B s_\gra w \tilde e_S} \subset \overline{B w \tilde e_S}.	\qedhere
\]
\end{proof}

\section{The Bruhat order on $\calN_2$.}

We now characterize the Bruhat order among the $B$-orbits in $\calN_2$. Notice that by Proposition \ref{prop:B-orbits-in-X} we have the decomposition
$$
	\calN_2 = \bigcup_{\height(S) \leq 2} Be_S
$$

Suppose that $S \subset \Phi$ is strongly orthogonal with $\height(e_S) = 2$. If $\gra$ is a real descent for $\grs_{\wh S}$, it follows by Proposition \ref{prop:discese interne} that there are uniquely determined $\grg_1, \grg_2 \in S$ such that $\gra = \tfrac{1}{2}(\grg_1-\grg_2)$: we denote in this case
$$
	S_\gra = \big(S \setminus \{\grg_1, \grg_2\}\big) \cup \{\grg_2+\gra\}.
$$
Thus, if $\gra$ is a descent for $\grs_{\wh S}$, we define
$$
	\calF_\gra(S) = \left\{
	\begin{array}{cl}
		s_\gra (S) & \text{ if $\gra$ is complex for $\grs_{\wh S}$}, \\
		S_\gra & \text{ if $\gra$ is real for $\grs_{\wh S}$.}
	\end{array} \right.
$$

\begin{lemma}	\label{lemma:discese-sotto}
Let $S \subset \Phi$ be strongly orthogonal with $\height(e_S) = 2$ and let $\gra \in \grD$ be a descent for $\grs_{\wh S}$, then $\overline{Be_S}$ is $P_\gra$-stable. If moreover $S' = \calF_\gra(S)$, then $Be_{S'} \subset P_\gra e_S \subset  \overline{B e_S}$ and $\grs_{\wh S'} = s_\gra \circ \grs_{\wh S}$.
\end{lemma}

\begin{proof}
Denote $X = \overline{Ge}$ and let $h$ be the dominant characteristic of $Ge_S$. Let moreover $P$ be the standard parabolic subgroup of $G$ defined by $h$, let $\goa \subset \gop^u$ be the associated abelian ideal of $\gob$ and let $\Psi \subset \Phi^+$ be the set of roots occurring in $\goa$. 

Let $(w,R) \in W^P \times \Ort(\Psi)$ be the admissible pair such that $S = w(R)$. Then by Proposition \ref{prop:discese esterne} the root $\gra$ is a descent for $\grs_{\wh S} = \grs_{w(\wh R)}$ if and only if it is a descent for $(w,R)$. Setting $\wt X = G \times_P \goa$, we see by Lemma \ref{lemma:discese} that $\overline{Bw \tilde e_R} \subset \wt X$ is $P_\gra$-stable. As the resolution $\phi : \wt X \rightarrow X$ is closed, it follows that $\overline{Be_S} = \phi(\overline{Bw \tilde e_R})$ is also $P_\gra$-stable.

The last claim follows from Remark \ref{oss:figli}, since  $S'$ is the image of $\calF_\gra(w,R)$ via the natural contraction $(w,S) \mapsto w(S)$.
\end{proof}

\begin{theorem}	\label{teo:bruhat3}
Let $R,S \subset \Phi$ be strongly orthogonal  and suppose that $\height(e_R) = \height(e_S) = 2$. Then $B e_R \subset \overline {B e_S}$ if and only if $\grs_{\wh R} \leq \grs_{\wh S}$.
\end{theorem}

\begin{proof}
One implication follows from Proposition \ref{prop:bruhat-involuzioni}.
To prove the converse, we adapt the proof of Theorem \ref{teo:bruhat1}.

Denote $X = \overline{Ge_S}$ and let $h$ be the dominant characteristic of $Ge_S$. Let moreover $P$ be the standard parabolic subgroup defined by $h$, let $\goa \subset \gop^u$ be the associated abelian ideal of $\gob$ and let $\Psi \subset \Phi^+$ be the set of roots occurring in $\goa$. By Proposition \ref{prop:B-orbits-in-X}, there exists $w \in W^P$ such that $S\subset w(\Psi)$. Notice that $w$ is uniquely determined, since $e_S$ is in the open $G$-orbit of $X$. Equivalently, $(w, w^{-1}(S))$ is the admissible pair defined by $Be_S \subset Ge$. By induction on $\ell(w)$, we show that $Be_R \subset \overline{Be_S}$.

Suppose that $\ell(w) = 0$. Then $S \subset \Psi$, and by Proposition \ref{prop:bruhat-involuzioni2} we get $R \subset \Psi$ as well. In particular $e_R \in \goa$, therefore $Be_R \subset \overline{Be_S}$ by \cite[Theorem 5.3]{GMMP}.

Suppose now that $\ell(w) > 0$ and let $\gra \in \grD$ be such that $s_\gra w < w$.
Then $\gra$ is a complex descent for $\grs_{\wh S}$ by Proposition \ref{prop:discese esterne} ii),
and by Lemma \ref{lemma:discese-sotto} we have
$$B e_S \subset P_\gra e_{s_\gra(S)} \subset \ol{B e_S}.$$
Notice that the admissible pair associated to $s_\gra(S)$ is $(s_\gra w, w^{-1}(S))$.

Suppose first that $\gra$ is a descent for $\grs_{\wh R}$ and set $R' = \calF_\gra(R)$. Then by Lemma \ref{lemma:discese-sotto} we have $\grs_{\wh R'} = s_\gra \circ \grs_{\wh R}$ and 
$$B e_R \subset P_\gra e_{R'} \subset \ol{B e_R}.$$
Using Lemma \ref{lemma:par2} it is easy to see that $\grs_{\wh R'} \leq \grs_{s_\gra (\wh S)}$, thus we can apply the inductive assumption and we get
$$
	B e_R \subset P_\gra e_{R'} \subset \overline{P_\gra e_{s_\gra(S)}} \subset \overline{B e_S}.
$$

Suppose now that $\gra$ is not a descent for $\grs_{\wh R}$, then using Lemma \ref{lemma:par2} it is easy to see that $\grs_{\wh R} \leq \grs_{s_\gra (\wh S)}$. Thus by the inductive hypothesis we get
\[
	B e_R \subset \overline{B e_{s_\gra(S)}} \subset \overline{B e_S}.	\qedhere
\]
\end{proof}


\end{document}